\documentclass[12pt]{amsart}
\usepackage{amssymb,latexsym,amsthm}
\usepackage{amsmath,amsfonts,amssymb}
\usepackage{graphicx}
\usepackage{color}
\usepackage{mathrsfs}

\newcommand{\M}{\mathfrak M}
\newcommand{\nn}{\|\hspace{-0.45mm}|}

\newcommand{\cno}{\textcircled{\small 1}}
\newcommand{\cnt}{\textcircled{\small 2}}

\newcommand{\suppm}{\operatorname{supp}}
\newcommand{\suppmj}{\operatorname{supp}(\mu_j)}
\newcommand{\suppmthe}{\operatorname{supp}(\mu_\theta)}

\newcommand{\wbo}{\widetilde{B_1}}
\newcommand{\wbt}{\widetilde{B_2}}
\newcommand{\wbj}{\widetilde{B_j}}
\newcommand{\wb}{\widetilde{B}}

\newcommand{\Lip}{\operatorname{Lip}}
\newcommand{\lip}{\operatorname{lip}_{\alpha}}
\newcommand{\Log}{\operatorname{Log}}

\newcommand{\ch}{\operatorname{Ch}}

\newtheorem{theorem}{Theorem}
\newtheorem{lemma}[theorem]{Lemma}
\newtheorem{cor}[theorem]{Corollary}
\newtheorem{prop}[theorem]{Proposition}

\theoremstyle{definition}
\newtheorem{definition}[theorem]{Definition}
\newtheorem{example}[theorem]{Example}

\theoremstyle{remark}

\begin{document}
\author{
Osamu~Hatori
}
\address{
Department of Mathematics, Faculty of Science,
Niigata University, Niigata 950-2181, Japan
}
\email{hatori@math.sc.niigata-u.ac.jp
}

\author{
Shiho Oi
}
\address{
Niigata Prefectural Nagaoka High School, 3-14-1 Gakko-cho, Nagaoka City, Niigata Prefecture 940-0041, Japan.
}
\email{shiho.oi.pmfn20@gmail.com
}

\dedicatory
{Dedicated to Hiroyuki Takagi}

\title[Isometries on Banach algebras]
{Isometries on Banach algebras of vector-valued maps}




\begin{abstract}
We propose a unified approach to the study of isometries on algebras of vector-valued Lipschitz maps and those of continuously differentiable maps by means of the notion of admissible quadruples. We describe isometries on function spaces of some admissible quadruples that take values in unital commutative $C^*$-algebras. As a consequence we confirm the statement of \cite[Example 8]{jp} on Lipschitz algebras and show that isometries on such algebras indeed take the canonical form.
\end{abstract}
\maketitle
\section{Introduction}
A long tradition of inquiry seeks sufficient sets of conditions on a linear map $U$ between Banach spaces in order that $U$ preserves the distance of elements in the spaces. The most prominent result along these lines is the Banach--Stone theorem on a linear map on the space $C(Y)$ (resp. $C_{\mathbb R}(Y)$) of complex-valued (resp. real-valued) continuous functions on a compact Hausdorff space $Y$.
Researchers have derived extensions of the Banach--Stone theorem for several different settings. 
We refer the reader to \cite{fj1,fj2} for a survey of the topic. In this paper an isometry means a complex-linear isometry.

de Leeuw \cite{dl} probably initiated the study of isometries on the algebra of Lipschitz functions on the real line. Roy \cite{roy} studied isometries on the Banach space $\Lip(X)$ of Lipschitz functions on a compact metric space $X$, equipped with the norm $\|f\|=\max\{\|f\|_\infty, L(f)\}$, where $L(f)$ denotes the Lipschitz constant. 
Cambern \cite{c} has considered isometries on spaces of scalar-valued continuously differentiable functions $C^1([0,1])$ with norm given by $\|f\|=\max_{x\in [0,1]}\{|f(x)|+|f'(x)|\}$ for $f\in C^1([0,1])$ and determined a representation for the surjective isometries supported by such spaces. Jim\'enez-Vargas and Villegas-Vallecillos in \cite{amPAMS} have considered isometries of spaces of vector-valued Lipschitz maps on a compact metric space taking values in a strictly convex Banach space, equipped with the norm $\|f\|=\max\{\|f\|_\infty, L(f)\}$, see also \cite{amHouston}. Botelho and Jamison \cite{bjStudia2009} studied isometries on $C^1([0,1],E)$ with $\max_{x\in [0,1]}\{\|f(x)\|_E+\|f'(x)\|_E\}$.
See also \cite{rr,mw,amy,araduba,bfj,kos,bjz,rm,bjPositivity17,mt,kawar,kc1,kc12,lcmw,kkm,lpww,jlp}

From now on, and unless otherwise mentioned, $\alpha$ will be a real scalar in $(0,1)$. 
Jarosz and Pathak \cite{jp} studied a problem when an isometry on a space of continuous functions is a weighted composition operator. They provided a unified approach for certain function spaces including $C^1(X)$, $\Lip(X)$, $\lip(X)$ and $AC[0,1]$. On the other hand, 
 isometries on algebras of Lipschitz maps and continuously differentiable maps have often been studied independently. 
We propose a unified approach to the study of isometries on algebras  $\Lip(X,C(Y))$, $\lip(X,C(Y))$ and $C^1(K,C(Y))$, where $X$ is a compact metric space, $K=[0,1]$ or $\mathbb{T}$ (in this paper $\mathbb{T}$ denotes the unit circle on the complex plane),  and $Y$ is a compact Hausdorff space. We define an admissible quadruple of type L (see Definition \ref{aqL}) as a  common abstraction of Lipschitz algebras and algebras of continuously differentiable maps. We prove that a surjective isometry between admissible quadruple of type L is canonical (Theorem \ref{main}), in the sense that it is represented as a weighted composition operator. As corollaries we describe isometries on $\Lip(X,C(Y))$, $\lip(X,C(Y))$ and $C^1(K,C(Y))$ respectively (Corollaries \ref{isoLip}, \ref{c101}, \ref{c1t}). There is a variety of norms on $\Lip(X,C(Y))$, $\lip(X,C(Y))$ and $C^1(K,C(Y))$. In this paper we consider the norm of $\ell^1$-type; $\|F\|_{\infty(X\times Y)}+L(F)$ for $F\in \Lip(X,C(Y))$, $\|F\|_{\infty(X\times Y)}+L_\alpha(F)$ for $F\in \lip(X,C(Y))$ and $\|F\|_{\infty(K\times Y)}+\|F'\|_{\infty(K \times Y)}$ for $F\in C^1(K,C(Y))$. With these norms $\Lip(X,C(Y))$, $\lip(X,C(Y))$ and $C^1(K,C(Y))$ are commutative Banach algebras respectively. 

Jarosz and Pathak exhibited in \cite[Example 8]{jp} that a surjective isometry on $\Lip(X)$ and $\lip(X)$ of a compact metric space $X$ with respect to the norm $\|\cdot\|_\infty+ L_\alpha(\cdot)$ is canonical. 
There seem to be a confusion of the status of the result and it would be appropriate to clarify the current situation. After the publication of \cite{jp} some authors expressed their suspicion about the argument there and the validity of the statement there had not been confirmed when the authors of \cite{lpww} pointed out a gap by referring the comment of Weaver \cite[p. 243]{wea}. While Weaver in \cite{wea} pointed out that the argument of \cite{jp} failed on p.200 in which the norm $\max\{\|\cdot\|_\infty, L(\cdot)\}$ was studied, he did not seem to have stated explicitly that the argument in the Example 8 contained a flaw.

The authors of the present paper find it difficult to follow the argument given in the Example 8. Besides non-substantial typos, the well-definedness of the map $\Psi_\vartheta:\operatorname{ext}B^*\to \operatorname{ext}B^*$ (\cite[p. 205, line 8]{jp}), where  $\operatorname{ext}B^*$ is the set of all extreme points in the closed unit ball of the dual space of $B=\Lip_{\alpha'}(Y)$ given by $\Psi_\vartheta(\gamma \delta_{(y,\omega,\beta)})=\gamma \delta_{(y,\omega, e^{i\vartheta}\beta)}$ (note that the formula on the line 9 of \cite[p. 205]{jp} reads  in this way) seems to require further explanation. On the other hand Corollary \ref{JPOK} of this paper confirms the statement of \cite[Example 8]{jp}. Our proof uses a similar but slightly different vein than that of Jarosz-Pathak's argument.

The main result in this paper is Theorem \ref{main}, which gives the form of a surjective isometry $U$ between admissible quadruples of type L. The proof of the necessity of the isometry in Theorem \ref{main} comprises several steps. We give an outline of the proof. The crucial part of the proof of Theorem \ref{main} is to prove that $U(1)=1\otimes h$ for an $h\in C(Y_2)$ with $|h|=1$ on $Y_2$ (Proposition \ref{absolute value 1}). To prove Proposition \ref{absolute value 1} we apply Choquet's theory with measure theoretic arguments (Lemmata \ref{1},\ref{2}). By Proposition \ref{absolute value 1} we have that $U_0=(1\otimes \bar{h})U$ is a surjective isometry fixing the unit. Then by applying a theorem of Jarosz \cite{ja} we see that $U_0$ is also an isometry with respect to the supremum norm. By the Banach--Stone theorem $U_0$ is an algebra isomorphism and applying \cite{hots} we see that $U_0$ is a composition operator of type BJ. 
\section{Preliminaries with Definitions and Basic Results
}
\subsection{Algebras of Lipschitz maps and  continuously differentiable maps}
Let $Y$ be a compact Hausdorff space. 
Let $E$ be a complex Banach space. The space of all $E$-valued  continuous maps on $Y$ is denoted by $C(Y,E)$. When $E={\mathbb C}$, $C(Y,E)$ is  abbreviated $C(Y)$. The space of all real-valued continuous functions on $Y$ is denoted by $C_{\mathbb R}(Y)$.
For a subset $K$ of $Y$, the supremum norm on $K$ is  $\|F\|_{\infty(K)}=\sup_{x\in K}\|F(x)\|_E$ for $F\in C(Y,E)$. 
When no confusion will result we omit the subscript $K$ and write only $\|\cdot\|_{\infty}$. 
Let $X$ be a compact metric space and $0<\alpha\le 1$. For $F\in C(X,E)$, put 
\[
L_\alpha(F)=\sup_{x\ne y}\frac{\|F(x)-F(y)\|_E}{d(x,y)^\alpha},
\]
which is called an $\alpha$-Lipschitz number of $F$, or just a Lipschitz number of $F$. When $\alpha=1$ we  omit the subscript $\alpha$ and write only $L(F)$. The space of all $F\in C(X,E)$ such that $L_\alpha(F)<\infty$ is denoted by $\Lip_\alpha(X,E)$. When $\alpha=1$ the subscript is omitted and it is written as $\Lip(X,E)$. 

When $0<\alpha<1$ the closed subspace
\begin{multline*}
\lip(X,E)
\\
=\{F\in \Lip_\alpha(X,E):\text{$\lim_{x\to x_0}\frac{\|f(x_0)-f(x)\|_E}{d(x_0,x)^\alpha}=0$ for every $x_0\in X$}\}
\end{multline*}
of $\Lip_\alpha(X,E)$ is called a little Lipschitz space.  In this paper the norm $\|\cdot\|$ of $\Lip_\alpha(X,E)$ (resp. $\lip(X,E)$) is defined by 
\[
\|F\|=\|F\|_{\infty(X)}+L_\alpha(F), \quad F\in \Lip_\alpha(X,E)\,\, \text{(resp. $\lip(X,E)$)}.
\]
Note that if $d(\cdot,\cdot)$ is a metric, then so is $d(\cdot,\cdot)^\alpha$, and is denoted by $d^\alpha$ which is called a H\"older metric. 
For a compact metric space $(X,d)$, 
$\Lip_\alpha((X,d), E)$ is isometrically isomorphic to $\Lip((X,d^\alpha),E)$. 
In this paper we mainly concern with $E=C(Y)$. In this case $\Lip_\alpha(X,C(Y))$ and $\lip(X,C(Y))$ are unital semisimple commutative Banach algebras with $\|\cdot\|$.  When $E={\mathbb C}$ $\Lip(X,{\mathbb C})$ (resp. $\lip(X,{\mathbb C})$) is abbreviated to $\Lip(X)$ (resp. $\lip(X)$).
There are a variety of complete norms other than $\|\cdot\|$. For example $\|\cdot\|_{\max}=\max\{\|\cdot\|_{\infty}, L_\alpha(\cdot)\}$ is such a norm, but  it fails to be submultiplicative. Hence $\Lip_\alpha(X,C(Y))$ and $\lip(X,C(Y))$ need not be Banach algebras with respect to the norm $\|\cdot\|_{\max}$.

Let $F\in C(K,C(Y))$ for $K=[0,1]$ or ${\mathbb T}$. We say that $F$ is continuously differentiable if there exists $G\in C(K,C(Y))$ such that
\[
\lim_{K\ni t\to t_0}\left\|\frac{F(t_0)-F(t)}{t_0-t}-G(t_0)\right\|_{\infty(Y)}=0
\]
for every $t_0\in K$. We denote $F'=G$. 
Put $C^1(K,C(Y))=\{F\in C(K,C(Y)):\text{$F$ is continuously differentiable}\}$. Then $C^1(K,C(Y))$ with norm $\|F\|=\|F\|_\infty+\|F'\|_\infty$ is a unital semisimple commutative Banach algebra. 
If $Y$ is singleton we may suppose that $C(Y)$ is isometrically isomorphic to ${\mathbb C}$ and we abbreviate $C^1(K,C(Y))$ by $C^1(K)$.

By identifying $C(X, C(Y))$ with $C(X\times Y)$ we may assume that $\Lip(X,C(Y))$ is a subalgebra of $C(X\times Y)$ by the correspondence
\[
F\in \Lip(X,C(Y)) \leftrightarrow ((x,y)\mapsto (F(x))(y))\in C(X\times Y).
\]
Throughout the paper we may assume that 
\begin{equation*}
\begin{split}
&\Lip(X,C(Y))\subset C(X\times Y), \\
&\lip(X,C(Y))\subset C(X\times Y), \\
&C^1(K,C(Y))\subset C(K\times Y).
\end{split}
\end{equation*}
We say that a subset $S$ of $C(Y)$ is point separating if $S$ separates the points of $Y$.
Suppose that $B$ is a unital point separating subalgebra of $C(Y)$ equipped with a Banach algebra norm. Then $B$ is semisimple because 
$\{f\in B:f(x)=0\}$ is a maximal ideal of $B$ for every $x\in X$ and the Jacobson radical of $B$ vanishes. 
 The unit of $B$ is denoted by $1_B$. When no confusion will result we omit the subscript $B$ and write simply as $1$. The maximal ideal space of $B$ is denoted by $M_B$. 
\begin{definition}
We say that $B$ is inverse-closed if $f\in B$ with $f(y)\ne 0$ for every $y\in Y$ implies $f^{-1}\in B$. 
We say that $B$ is natural if the map $e:Y\to M_B$ defined by $y\mapsto \phi_y$, where $\phi_y(f)=f(y)$ for every $f\in B$, is bijective. We say that $B$ is self-adjoint if $B$ is natural and satisfies that $f\in B$ implies that $\bar{f}\in B$ for every $f\in B$, where $\bar{\cdot}$ denotes the complex conjugation on $Y=M_B$. 
\end{definition}
Note that conjugate closedness of $B$ (that is $f\in B$ implies $\bar{f}\in B$) needs not imply the self-adjointness of $B$. 
\begin{prop}\label{el}
Let $Y$ be a compact Hausdorff space. Suppose that $B$ is a unital point separating subalgebra of $C(Y)$ equipped with a Banach algebra norm. If $B$ is dense in $C(Y)$ and inverse-closed, then $B$ is natural.
\end{prop}
\begin{proof}
Suppose that $e:Y\to M_B$ is not surjective. Then there exists $\phi\in M_B$ such that for every $y\in Y$ there exists $f_y\in B$ with $\phi(f_y)=0$ such that $f_y(y)=1$. As $Y$ is compact, there exists a finite number of $f_1,\dots, f_n\in B$ with $\phi(f_j)=0$ for $j=1,\dots, n$ such that $\sum_{j=1}^n|f_j|^2>0$ on $Y$. Since $B$ is uniformly dense in $C(Y)$ there exist $g_1,\dots, g_n\in B$ such that $\sum_{j=1}^nf_jg_j>0$ on $Y$. As $B$ is inverse-closed, there exists $h\in B$ such that $h\sum_{j=1}^nf_jg_j=1_B$. As $\phi(f_j)=0$ for $j=1,\dots, n$ we have $0=\phi(h\sum_{j=1}^nf_jg_j)=\phi(1_B)=1$, which is a contradiction.
\end{proof}
\begin{cor}\label{eell}
The unital Banach algebras $\Lip(X)$ and $\Lip(X,C(Y))$ with $\|\cdot\|_\infty+L(\cdot)$ are point separating and self-adjoint. For $0<\alpha<1$ the unital Banach algebras $\lip(X)$ with $\|\cdot\|_\infty+L_\alpha(\cdot)$ and $\lip(X,C(Y))$ with $\|\cdot\|_\infty+L_\alpha(\cdot)$ are point separating and self-adjoint. For $K=[0,1]$ and ${\mathbb T}$, the unital Banach algebras $C^1(K)$ with $\|\cdot\|_\infty+\|\cdot'\|_\infty$ and $C^1(K, C(Y))$ with $\|\cdot\|_\infty+\|\cdot'\|_\infty$ are point separating and self-adjoint.
\end{cor}
\begin{proof}
The Lipschitz algebra $\Lip(X)$ is a unital point separating subalgebra of $C(X)$ equipped with a Banach algebra norm $\|\cdot\|_\infty+L(\cdot)$. As $\Lip(X)$ is conjugate closed, the Stone-Weierstrass theorem asserts that $\Lip(X)$ is uniformly dense in $C(X)$. Thus it is natural by Proposition \ref{el}, and, self-adjoint. 
In a similar way to that for $\Lip(X)$ we infer that $\Lip(X,C(Y))$ is self-adjoint. 

Suppose that $0<\alpha<1$. Then we see that $\lip(X)$ separates the points of $X$. (Let $x,y$ be different points in $X$. Put $f:X\to {\mathbb C}$ by $f(\cdot)=d(\cdot,y)$. By a simple calculation we infer that $f\in \lip(X)$ and $f(x)\ne f(y)$.)  In the same way as above we see that $\lip(X)$ and $\lip(X,C(Y))$ are natural, hence self-adjoint.

Let $K=[0,1]$ or $K={\mathbb T}$.
In the same way as above we see that $C^1(K)$ is self-adjoint.  In the same way as above $C^1(K,C(Y))$ is self-adjoint.
\end{proof}
\subsection{Admissible quadruples of type L} 
An admissible quadruple was defined by Nikou and O'Farrell in \cite{no} (see also a comment just after Definition 2.2 in  \cite{hots}). 
The definition is  little complicated and we adopt a simpler definition that is sufficient for our purpose. For a detailed account of admissible quadruples see 
\cite{no} and \cite{hots}. 
Let $X$ and $Y$ be compact Hausdorff spaces. 
For functions $f\in C(X)$ and $g\in C(Y)$, let $f\otimes g\in C(X\times Y)$ be the function defined by $f\otimes g(x,y)=f(x)g(y)$, and for a subspace $E_X$ of $C(X)$ and a subspace $E_Y$ of $C(Y)$, let
\[
E_X\otimes E_Y=\left\{\sum_{j=1}^nf_j\otimes g_j: n\in {\mathbb N},\,\,f_j\in E_X,\,\,g_j\in E_Y\right\}.
\]
An admissible quadruple $(X, C(Y), B, \widetilde{B})$  in this paper is defined as follows. 
\begin{definition}\label{aqL}
Let $X$ and $Y$ be compact Hausdorff spaces. 
Let $B$ and $\wb$ be  unital point separating subalgebras of $C(X)$ and  $C(X\times Y)$ equipped with Banach algebra norms respectively which satisfy
\[
B\otimes C(Y)\subset \widetilde{B},\,\,\{F(\cdot, y):F\in \widetilde{B},\,\,y\in Y\}\subset B.
\]
We say that $(X,C(Y),B,\widetilde{B})$ is an admissible quadruple of type L if the following conditions are satisfied.
\begin{itemize}
\item[$\cno$]
The algebras $B$ and  $\widetilde{B}$ are self-adjoint.
\item[$\cnt$]
There exists a compact Hausdorff space $\mathfrak{M}$ and a complex-linear operator $D:\wb\to C(\mathfrak{M})$ such that 
\[D(\widetilde{B}\cap C_{\mathbb R}(X\times Y))\subset C_{\mathbb R}(\mathfrak{M})
\]
and also
\begin{itemize}
\item[(1)]
the norm $\|\cdot\|$ on $\widetilde{B}$ satisfies
\[
\|F\|=\|F\|_{\infty(X\times Y)}+\|D(F)\|_{\infty(\mathfrak{M})},\quad F\in \widetilde{B},
\]
\item[(2)]
$\operatorname{ker}D=1_B\otimes C(Y)$,
\item[(3)]
$\|D((1_B\otimes g)F)\|_{\infty(\mathfrak{M})}=\|D(F)\|_{\infty(\mathfrak{M})}$ for every $F\in \wb$ and $g\in C(Y)$ such that $|g|=1$ on $Y$.
\end{itemize}
\end{itemize}
\end{definition}
It will be appropriate to make a few comments on the above definition. 
First we do not assume that $D(\widetilde{B})$ is point separating.
Next $B$ and $\widetilde{B}$ are semisimple since they are point separating. 
For a point $x\in X$ define $e_x:\widetilde{B}\to C(Y)$ by $e_x(F)=F(x,\cdot)$ for every $F\in \widetilde{B}$. 
A theorem of \v Silov (see \cite[Theorem 3.1.11]{pal}) states that the map $e_x:\widetilde{B}\to C(Y)$ is automatically continuous for every $x\in X$ since $C(Y)$ is semisimple.
Hence it is straightforward to check that an admissible quadruple of type L is in fact  an admissible quadruple defined by Nikou and O'Farrell in \cite{no} (see also \cite{hots}). In particular if $X$ is a compact metric space, then $(X, C(Y), \Lip(X), \Lip(X,C(Y)))$, $(X,C(Y), \lip(X), \lip(X,C(Y)))$ and $(K, C(Y), C^1(K), C^1(K,C(Y)))$ for $K=[0,1],{\mathbb T}$ are admissible quadruples of type L. See Section \ref{example}.

We define a seminorm $\nn\cdot\nn$ on $\wb$ by $\nn F\nn=\|D(F)\|_{\infty(\mathfrak{M})}$ for $F\in \wb$. 
Note that $\nn\cdot\nn$ is one-invariant in the sense of Jarosz \cite{ja} ($\nn F\nn=\nn F+1_{\wb}\nn$ for every $F\in \wb$) since $1_{\wb}=1_B\otimes 1_{C(Y)}$ and $D(1_{\wb})=0$. 
The norm $\|\cdot\|=\|\cdot\|_\infty+\nn\cdot\nn$ is a $p$-norm (see \cite[p.67]{ja}).
\subsection{Preliminaries on measures} 
We recall some basic properties of regular Borel measures for the convenience of the readers. As the authors could not find appropriate references, we exhibit the properties in Lemmata \ref{0.1}, \ref{0.2} and \ref{0.3}. 
In Lemmata \ref{0.1} and \ref{0.2}, $X$ is a compact Hausdorff space and $\mu$ is a Borel probability measure (a positive measure on the $\sigma$-algebra of Borel sets whose total measure is $1$). For a non-empty Borel subset $S$ of $X$, $\mu|S$ denotes the measure on $S$ which is the restriction of $\mu$; $\mu|S(E)=\mu(E)$ for a Borel set $E\subset S$. Recall that the support of $\mu$ is the set defined by 
\[
\suppm \mu=\{x\in X: \text{$\mu(U)>0$ for every open neighborhood $U$ of $x$}\}.
\]
\begin{lemma}\label{0.1}
Let $K$ be a non-empty compact subset of $X$ and $f\in C(X)$. Assume that $f\le c$ on $K$ for a constant $c>0$. If
\[
\int_K f d\mu=c\mu(K),
\]
then $\suppm(\mu|K)\subset f^{-1}(c)\cap K$.
\end{lemma}
\begin{proof}
Let $x\in \suppm(\mu|K)$. Then $x\in K$ by the definition of the support of $\mu|K$. Suppose that $f(x)\ne c$. As $f\le c$ on $K$, we have $f(x)<c$. Since $f|K$ is continuous on $K$, there exists an open neighborhood $U$ of $x$ relative to $K$ such that $f<(f(x)+c)/2$ on $U$. As $x\in \suppm(\mu|K)$ we have that $\mu(U)>0$. Then
\begin{equation*}
\begin{split}
\int_Kfd\mu
&=\int_Ufd\mu + \int_{K\setminus U}fd\mu \\
&\le \frac{f(x)+c}{2}\mu(U) +c\mu(K\setminus U) \\
&= c\mu(K)-\frac{c-f(x)}{2}\mu(U)<c\mu(K),
\end{split}
\end{equation*}
which is a contradiction proving that $f(x)=c$. Thus we conclude that $\suppm(\mu|K)\subset f^{-1}(c)\cap K$.
\end{proof}
\begin{lemma}\label{0.2}
Suppose that $K_1$ and $K_2$ are non-empty compact subsets of $X$. Then
\[
\suppm (\mu|K_1) \cup \suppm (\mu|K_2) =\suppm( \mu|(K_1\cup K_2)).
\]
\end{lemma}
\begin{proof}
Suppose that $x\in \suppm(\mu|K_1)$. Let $G$ be an arbitrary open neighborhood of $x$ relative to $K_1\cup K_2$. Then there is an open set $\tilde G$ in $X$ with $\tilde{G}\cap (K_1\cup K_2)=G$. Then $\tilde{G}\cap K_1$ is an open neighborhood of $x$ relative to $K_1$ and $G=\tilde{G}\cap(K_1\cup K_2)\supset \tilde{G}\cap K_1$. As $x\in \suppm(\mu|K_1)$ we have $0<\mu(\tilde{G}\cap K_1)\le \mu(G)$. Since $G$ is arbitrary we conclude that $x\in \suppm(\mu|(K_1\cup K_2))$; that is $\suppm(\mu|K_1)\subset\suppm(\mu|(K_1\cup K_2))$. In the same way we have $\suppm(\mu|K_2)\subset\suppm(\mu|(K_1\cup K_2))$. Thus we have $\suppm(\mu|K_1) \cup \suppm(\mu|K_2)\subset\suppm(\mu|(K_1\cup K_2))$.

Suppose conversely that $x\in \suppm(\mu|(K_1\cup K_2))$. Then $x\in K_1\cup K_2$. Suppose that $x\not\in \suppm(\mu|K_1)\cup \suppm(\mu|K_2)$. First we consider the case that $x\in K_1\cap K_2$. Then there is an open neighborhood $G_1$ of $x$ relative to $K_1$ and an open neighborhood $G_2$ of $x$ relative to $K_2$ such that $\mu(G_1)=\mu(G_2)=0$ since we have assumed that $x\not\in \suppm(\mu|K_1)\cup \suppm(\mu|K_2)$. There exist open sets $\tilde{G_1}$ and $\tilde{G_2}$ in $X$ such that $\tilde{G_1}\cap K_1=G_1$ and $\tilde{G_2}\cap K_2=G_2$. Put $\tilde{G}=\tilde{G_1}\cap \tilde{G_2}$. Then $\tilde{G}$ is an open set in $X$ and $x\in \tilde{G}$. Then $\tilde{G}\cap(K_1\cup K_2)$ is an open neighborhood of $x$ relative to $K_1\cup K_2$ and 
\[
\tilde{G}\cap (K_1\cup K_2)=(\tilde{G}\cap K_1)\cup(\tilde{G}\cap K_2)
\subset 
(\tilde{G_1}\cap K_1)\cup (\tilde{G_2}\cap K_2)=G_1\cup G_2.
\]
Then 
\[
0\le \mu(\tilde{G}\cap (K_1\cup K_2))\le \mu(G_1\cup G_2)
\le \mu(G_1)+\mu(G_2)=0,
\]
so that $\mu(\tilde{G}\cap (K_1\cup K_2))=0$, which is a contradiction since $x\in \suppm(\mu|(K_1\cup K_2))$. Next we consider the case where $x\in K_1$ and $x\not\in K_2$. Then there exists an open neighborhood $G_1$ of $x$ relative to $K_1$ with $\mu(G_1)=0$ since we have assumed that $x\not\in\suppm(\mu|K_1)$. There exists an open set $\tilde{G_1}$ in $X$ such that $\tilde{G_1}\cap K_1=G_1$. Since $x\not\in K_2$ we infer that $\tilde{G_1}\cap K_2^c$ is an open neighborhood of $x$ in $X$. Then $(\tilde{G_1}\cap K_2^c)\cap(K_1\cup K_2)$ is an open neighborhood of $x$ relative to $K_1\cup K_2$ and 
\[
(\tilde{G_1}\cap K_2^c)\cap(K_1\cup K_2)=\tilde{G_1}\cap K_2^c\cap K_1\subset \tilde{G_1}\cap K_1=G_1.
\]
As $(\tilde{G_1}\cap K_2^c)\cap(K_1\cup K_2)$ is an open neighborhood of $x$ relative to $K_1\cup K_2$, we infer that 
 $0<\mu((\tilde{G_1}\cap K_2^c)\cap(K_1\cup K_2))$ since $x\in \suppm(\mu|(K_1\cup K_2))$. On the other hand $(\tilde{G_1}\cap K_2^c)\cap(K_1\cup K_2)\subset G_1$ assures that 
\[
0< \mu((G_1\cap K_2^c)\cap(K_1\cup K_2))\le \mu(G_1)=0,
\]
which is a contradiction. In the same way we derive a contradiction for the case where $x\not\in K_1$ and $x\in K_2$. Therefore we conclude that $x\in \suppm(\mu|K_1)\cup \suppm(\mu|K_2)$.
\end{proof}
We assume the regularity for the measure $\mu$ in Lemma \ref{0.3}. If $\mu$ is a regular Borel probability measure on a compact Hausdorff space $Y$, then for any Borel set $S$ in $Y\setminus \suppm(\mu)$ we have $\mu(S)=0$. Indeed the regularity of $\mu$ assures that $\mu(S)$ is approximated arbitrarily closely  by $\mu(E)$ for a compact subset $E\subset S$. Since $S\cap \suppm(\mu)=\emptyset$, we use the compactness to cover $E$ by a finitely many open sets with measure zero. This implies $\mu(E)=0$ and thus $\mu(S)=0$.

\begin{lemma}\label{0.3}
Let $Y$ be a compact Hausdorff space and let $K$ be a non-empty compact subset of $Y$ and let 
$\mu$ be a regular Borel probability measure on $Y\times \mathbb{T}$. 
Let $g\in C_{\mathbb{R}}(Y)$ such that $|g|\le c$ on $K$ for some $c>0$. Suppose that there exists $\gamma_0\in \mathbb{T}$ such that
\[
\int_{K\times\mathbb{T}}\gamma g(y)
d\mu(y,\gamma)=\gamma_0c \mu(K\times \mathbb{T}).
\]
Then we have the inclusion
\begin{multline*}
\suppm (\mu|K\times \mathbb{T})
\\
\subset 
\left\{(g^{-1}(c)\cap K)\times \{\gamma_0\}\right\}
\cup
\left\{(g^{-1}(-c)\cap K)\times \{-\gamma_0\}\right\}.
\end{multline*}
\end{lemma}
\begin{proof}
As $|\gamma g|=|g|\le c$ on $K\times \mathbb{T}$ we have
\[
c\mu (K\times \mathbb{T})=\left|\int_{K\times\mathbb{T}}\gamma g(y)d\mu\right|\le \int_{K\times \mathbb{T}}|g(y)|d\mu\le c\mu(K\times \mathbb{T}),
\]
hence $\int_{K\times \mathbb{T}}|(g\otimes 1_{C(\mathbb{T})})(y,\gamma)|d\mu=\int_{K\times \mathbb{T}}|g(y)|d\mu= c\mu(K\times \mathbb{T})$. 
Letting $|g\otimes 1_{C(\mathbb{T})}|$ be 
the function 
$f$ 
 and $K\times \mathbb{T}$ be the compact set $K$ of Lemma \ref{0.1} respectively we have 
\[
\suppm(\mu|K\times \mathbb{T})\subset (|g\otimes 1_{C({\mathbb T})}|^{-1}(c))\cap (K\times\mathbb{T})=
(|g|^{-1}(c)\cap K)\times \mathbb{T}.
\]
As $g$ is a real-valued function we infer by a simple calculation that 
\[
|g|^{-1}(c)=g^{-1}(c)\cup g^{-1}(-c).
\]
Put $K_1=g^{-1}(c)$ and $K_2=g^{-1}(-c)$. As $c>0$, we have $K_1\cap K_2=\emptyset$. Then
\begin{multline}\label{o1}
\suppm(\mu|K\times \mathbb{T})
\subset ((K_1\cup K_2)\cap K)\times\mathbb{T}\\
=(K_1\cap K)\times \mathbb{T} \cup (K_2\cap K)\times \mathbb{T}.
\end{multline}
As $\mu$ is regular, we have that 
\[
\mu(K\times \mathbb{T}\setminus[(K_1\cap K)\times \mathbb{T} \cup (K_2\cap K)\times \mathbb{T}])=0.
\]
It follows that 
\begin{multline*}
\gamma_0 c \mu (K\times \mathbb{T})=
\int_{K\times\mathbb{T}}\gamma g(y)d\mu \\
=\int_{(K_1\cap K)\times\mathbb{T}}\gamma g(y)d\mu +
\int_{(K_2\cap K)\times\mathbb{T}}\gamma g(y)d\mu \\
=c\int_{(K_1\cap K)\times\mathbb{T}}\gamma d\mu-c\int_{(K_2\cap K)\times\mathbb{T}}\gamma d\mu.
\end{multline*}
Thus we have
\begin{equation}\label{sub1}
\mu(K\times \mathbb{T})=\int_{(K_1\cap K)\times \mathbb{T}}\overline{\gamma_0}\gamma d\mu-\int_{(K_2\cap K)\times\mathbb{T}}\overline{\gamma_0}\gamma d\mu.
\end{equation}
Put $M_1=\int_{(K_1\cap K)\times \mathbb{T}}1d\mu$ and $M_2=\int_{(K_2\cap K)\times \mathbb{T}}1d\mu$. As $\mu$ is regular and $K_1\cap K_2=\emptyset$ we have
\begin{equation}\label{sub2}
M_1+M_2= \int_{((K_1\cup K_2)\cap K)\times \mathbb{T}}1d\mu=\int_{K\times \mathbb{T}}1d\mu=\mu(K\times \mathbb{T}).
\end{equation}
Put
\[
\int_{(K_1\cap K)\times\mathbb{T}}\overline{\gamma_0}\gamma d\mu=e^{i\delta_1}N_1, \quad 
\int_{(K_2\cap K)\times\mathbb{T}}\overline{\gamma_0}\gamma d\mu=e^{i\delta_2}N_2,
\]
where $N_1, N_2\ge 0$ and $\delta_1,\delta_2\in {\mathbb R}$. We may assume that $e^{i\delta_1}=1$ if $N_1=0$ and $e^{i\delta_2}=-1$ if $N_2=0$. Note that $N_1\le M_1$ and $N_2\le M_2$. By \eqref{sub1} and \eqref{sub2} we obtain 
\[
M_1+M_2=e^{i\delta_1}N_1-e^{i\delta_2}N_2.
\]
 Then by a simple calculation we have that $e^{i\delta_1}=-e^{i\delta_2}=1$, $N_1=M_1$, and $N_2=M_2$, that is, 
\[
\int_{(K_1\cap K)\times\mathbb{T}}\overline{\gamma_0}\gamma d\mu =\mu((K_1\cap K)\times\mathbb{T}), \quad 
\int_{(K_2\cap K)\times\mathbb{T}}-\overline{\gamma_0}\gamma d\mu =\mu((K_2\cap K)\times\mathbb{T}).
\]
Then 
\begin{equation}\label{sub3}
\mu((K_1\cap K)\times\mathbb{T})=\operatorname{Re}\int_{(K_1\cap K)\times \mathbb{T}}\overline{\gamma_0}\gamma d\mu=
\int_{(K_1\cap K)\times\mathbb{T}}\operatorname{Re}\overline{\gamma_0}\gamma d\mu,
\end{equation}
\begin{equation}\label{sub4}
\mu((K_2\cap K)\times\mathbb{T})=\operatorname{Re}\int_{(K_2\cap K)\times \mathbb{T}}-\overline{\gamma_0}\gamma d\mu=
\int_{(K_2\cap K)\times\mathbb{T}}\operatorname{Re}(-\overline{\gamma_0}\gamma) d\mu.
\end{equation}
Applying Lemma \ref{0.1} to \eqref{sub3} we infer that
\[
\operatorname{supp}(\mu|((K_1\cap K)\times\mathbb{T}))\subset
(K_1\cap K)\times \{\gamma_0\}
\]
since $\operatorname{Re}\overline{\gamma_0}\gamma  \le 1$ and $(\operatorname{Re}\overline{\gamma_0}\gamma)^{-1}(1)=\{\gamma_0\}$. 
In the same way we have by \eqref{sub4} that
\[
\operatorname{supp}(\mu|((K_2\cap K)\times\mathbb{T}))\subset
(K_2\cap K)\times \{-\gamma_0\}.
\]
By Lemma \ref{0.2} we have that 
\begin{multline}\label{o2}
\operatorname{supp}\Big(\mu|\big(((K_1\cup K_2)\cap K)\times\mathbb{T}\big)\Big)
\\
\subset
\big\{(K_1\cap K)\times\{\gamma_0\}\big\}\cup
\big\{(K_2\cap K)\times\{-\gamma_0\}\big\}.
\end{multline}
Since $\mu$ is regular, so is $\mu|(K\times\mathbb{T})$. Thus $\mu|(K\times\mathbb{T})$ is a regular Borel measure on $K\times\mathbb{T}$ such that $\operatorname{supp}(\mu|(K\times\mathbb{T}))\subset ((K_1\cup K_2)\cap K)\times\mathbb{T}$ by \eqref{o1}. Thus
\begin{multline*}
\operatorname{supp}(\mu|(K\times\mathbb{T}))=
\operatorname{supp}\Big(\big(\mu|(K\times\mathbb{T})\big)|\big(((K_1\cup K_2)\cap K)\times\mathbb{T}\big)\Big)\\
=\operatorname{supp}\Big(\mu|\big(((K_1\cup K_2)\cap K)\times\mathbb{T}\big)\Big),
\end{multline*}
hence the conclusion holds by \eqref{o2}.
\end{proof}
Lemma \ref{0.3} plays an essential role in the proof of Lemma \ref{2} which is a crucial lemma for the proof of Proposition \ref{absolute value 1}.
\section{Isometries on admissible quadruples of type L}
The main result of this paper is the following. 
\begin{theorem}\label{main}
Suppose that $(X_j, C(Y_j), B_j, \wbj)$ is an admissible quadruple of type L for $j=1,2$. 
Suppose that $U:\wbo\to \wbt$ is a surjective isometry. Then there exists $h\in C(Y_2)$ such that $|h|=1$ on $Y_2$, a continuous map $\varphi:X_2\times Y_2\to X_1$ such that $\varphi(\cdot,y):X_2\to X_1$ is a homeomorphism for each $y\in Y_2$, and a homeomorphism $\tau:Y_2\to Y_1$ which satisfy
\[
U(F)(x,y)=h(y)F(\varphi(x,y),\tau(y)),\qquad (x,y)\in X_2\times Y_2
\]
for every $F\in \wbo$.
\end{theorem}
In short a surjective isometry between admissible quadruples of type L is canonical, that is, a weighted composition operator of a specific form: the homeomorphism $X_2\times Y_2 \to X_1\times Y_1, \,\,(x,y)\mapsto (\varphi(x,y),\tau(y))$ has the second coordinate that depends only on the second variable $y\in Y_2$. 
A composition operator induced by such a homeomorphism is said to be of type BJ in \cite{ho,hots} after the study of Botelho and Jamison \cite{bjRocky}. That every composition operator on an admissible quadruple  $(X, E, B, \widetilde{B})$ onto itself is of type BJ indicates that $B$ and $E$ are totally different Banach algebras.
\section{The form of $U(1_{\wb_1})$}
Throughout this section we assume 
that $U:\wbo \to \wbt$ is a surjective linear isometry satisfying all the hypotheses of Theorem \ref{main} without further mention. For the simplicity of the proof of Theorem \ref{main} we assume that $X_2$ is not a singleton in this section. Our main purpose in this section is to prove Proposition \ref{absolute value 1}, which is a crucial part of proof of Theorem \ref{main}.
\begin{prop}\label{absolute value 1}
There exists $h\in C(Y_2)$ with $|h|=1$ on $Y_2$ such that $U(1_{\wb_1})=1_{B_2}\otimes h$.
\end{prop}
Lemma \ref{2} is crucial for the proof of Proposition \ref{absolute value 1}.  We prove Lemma \ref{2} by applying Choquet's theory (\cite{ph}) which studies the extreme point of the dual unit ball of the space of continuous functions with the supremum norms. To apply the theory we first define an isometry
 from $\widetilde{B_j}$ into a uniformly closed space of complex-valued continuous functions. 
 Let $j=1, 2$. Define a map
\[
I_j:\wbj \to C(X_j\times Y_j\times \M_j\times \mathbb{T})
\]
by $I_j(F)(x,y,m,\gamma)=F(x,y)+\gamma D_j(F)(m)$ for $F\in \wbj$ and $(x,y,m,\gamma)\in X_j\times Y_j\times M_j\times \mathbb{T}$. (Recall that $\mathbb{T}$ is the unit circle in the complex plane.) As $D_j$ is a complex linear map, so is $I_j$. Let $S_j=X_j\times Y_j\times M_j\times \mathbb{T}$. For simplicity we just write $I$ and $D$ instead of $I_j$ and $D_j$ without causing any confusion. For every $F\in \wbj$ the supremum norm $\|I(F)\|_{\infty}$ on $S_j$ of $I(F)$ is written as  
\begin{equation*}
\begin{split}
\|I(F)\|_\infty 
& =\sup\{|F(x,y)+\gamma D(F)(m)|:(x,y,m,\gamma)\in S_j\}\\
& =\sup\{|F(x,y)|:(x,y)\in X_j\times Y_j\}\\
&\qquad 
+\sup\{|D(F)(m)|:m\in \M_j\}\\
&=\|F\|_{\infty(X_j\times Y_j)}+\|D(F)\|_{\infty(\M)}.
\end{split}
\end{equation*}
The second equality follows by an inspection that $\gamma$ runs through the whole $\mathbb{T}$. 
It follows that
\[
\|I(F)\|_{\infty}=\|F\|_{\infty}
+\|D(F)\|_{\infty}=
 \|F\|
\]
for every $F\in \wbj$. Since $0= \| D(1)\|_{\infty}$, we have $D(1)=0$ and  $I(1)=1$. Hence $I$ is a complex-linear isometry with $I(1)=1$. In particular, $I(\wbj)$ is a complex-linear closed subspace of $C(S_j)$ which contains $1$. In general $I(\wbj)$ needs not separate the points of $S_j$.

It follows from the definition in \cite{ph} of the Choquet boundary $\ch I(\wbt)$ of $I(\wbt)$, we see that a point $p=(x,y,m,\gamma)\in X_2\times Y_2\times\mathfrak{M}\times \mathbb{T}$ is in $\ch I(\wbt)$ if the point evaluation $\phi_p$ at $p$ is an extreme point of the state space, or equivalently $\phi_p$ is an extreme point of the closed unit ball $(I(\wbt))^*_1$ of the dual space $(I(\wbt))^*$ of $I(\wbt)$.
\begin{lemma}\label{1}
Suppose that 
$(x_0,y_0)\in X_2\times Y_2$ and ${\mathfrak U}$ is an open neighborhood of $(x_0,y_0)$. Then there exists a function $F_0=b_0\otimes f_0\in \wbt$ with $0\le b_0\le 1$ on $X_2$ and $0\le f_0\le 1$ on $Y_2$ such that $F_0(x_0,y_0)=1$  and $F_0<1/2$ on $X_2\times Y_2\setminus \mathfrak{U}$. Furthermore there exists a point $(x_c,y_c,m_c,\gamma_c)$ in the Choquet boundary for $I_2(\wbt)$ such that $(x_c,y_c)\in {\mathfrak U}\cap (b^{-1}_0(1)\times f^{-1}_0 (1))$ and $\gamma_cD(F_0)(m_c)=\|D(F_0)\|_\infty\ne 0$. 
\end{lemma}
\begin{proof}
Suppose that $\mathfrak{G}$ and $\mathfrak{H}$ are open neighborhoods of $x_0$ and $y_0$ respectively such that $\mathfrak{G\times H}\subset \mathfrak{U}$. Since $B_2$ is unital, self-adjoint and separates the points of $X_2$, the Stone-Weierstrass theorem asserts that $B_2$ is uniformly dense in $C(X_2)$. By the Urysohn's lemma there exists $v\in C(X_2)$ such that $0\le v\le 4/5$ on $X_2$, $v(x_0)=0$, and $v=4/5$ on $X_2\setminus \mathfrak{G}$. As $B_2$ is self-adjoint and uniformly dense in $C(X_2)$, there exists $u_1\in B_2\cap C_{\mathbb R}(X_2)$ such that $\|v-u_1\|_\infty<1/40$. Put $u=u_1-u_1(x_0)$. 
By a simple calculation we infer that $u\in B_2$ with $u(x_0)=0$ and $-1\le u\le 1$ on $X_2$ and $u^2>1/2$ on $X_2\setminus \mathfrak{G}$. Then $b_0=1-u^2\in B_2$, $0\le b_0\le 1=b_0(x_0)$ on $X_2$, and $b_0< 1/2$ on $X_2\setminus \mathfrak{G}$. We may suppose that $b_0$ is not constant as we assume that $X_2$ is not a singleton. In a similar way, there exists $f_0\in C(Y_2)$ with $0\le f_0\le 1=f_0(y_0)$ and $f_0<1/2$ on $Y_2\setminus \mathfrak{H}$. Put $F_0=b_0\otimes f_0$. Hence we have that $0\le F_0\le 1=F_0(x_0,y_0)$ and $F_0<1/2$ on $X_2\times Y_2\setminus \mathfrak{U}$. 
Since $B_2\otimes C(Y_2)\subset \wbt$ by Definition \ref{aqL}, 
we infer that $F_0\in \wbt$.

By Proposition 6.3 in \cite{ph} there exists $c=(x_c,y_c,m_c,\gamma_c)$ in the Choquet boundary for $I(\wbt)$ with
\[
\|I(F_0)\|_{\infty}=|I(F_0)(c)|.
\]
We see that 
\begin{multline}\label{sub5}
\|I(F_0)\|_{\infty}=|I(F_0)(c)|=|F_0(x_c,y_c)+\gamma_cD(F_0)(m_c)| \\
\le |F_0(x_c,y_c)|+|D(F_0)(m_c)|\le \|F_0\|_\infty + \|D(F_0)\|_\infty
=\|I(F_0)\|_{\infty}.
\end{multline}
As $0\le F_0\le 1=\|F_0\|_\infty$ we have by \eqref{sub5} that $F_0(x_c,y_c)=1=\|F_0\|_\infty$.  Thus $(x_c,y_c)\in {\mathfrak U}\cap (b^{-1}_0(1)\times f^{-1}_0(1))$. Applying that $F_0(x_c,y_c)=1$ and \eqref{sub5}, we also have that  $\gamma_cD(F_0)(m_c)=|D(F_0)(m_c)|=\|D(F_0)\|_\infty$.
As $b_0$ is not a constant function, we have $F_0=b_0\otimes f_0\not\in 1\otimes C(Y_2)=\operatorname{ker} D$. Hence we have $\|D(F_0)\|_\infty \ne 0$, so that $\|D(F_0)\|_\infty> 0$. As $F_0$ is real-valued, so is $D(F_0)$ by the condition $\cnt$ of Definition \ref{aqL}. 
Hence we see that  $\gamma_cD(F_0)(m_c)>0$ and $\gamma_c=1$ or $-1$.  
\end{proof}
Note that $\gamma_c=1$ if $D(F_0)(m_c)>0$ and $\gamma_c=-1$ if $D(F_0)(m_c)<0$.
\begin{lemma}\label{2}
Suppose that 
$(x_0,y_0)\in X_2\times Y_2$ and ${\mathfrak U}$ is an open neighborhood of $(x_0,y_0)$. Let $F_0=b_0\otimes f_0\in \wbt$ be a function such that $0\le b_0\le 1$ on $X_2$, $0\le f_0\le 1$ on $Y_2$, $F_0(x_0,y_0)=1$, and $F_0<1/2$ on $X_2\times Y_2\setminus \mathfrak{U}$. Let $(x_c,y_c,m_c,\gamma_c)$ be a point in the Choquet boundary for $I_2(\wbt)$ such that $(x_c,y_c)\in {\mathfrak U}\cap (b^{-1}_0(1)\times f^{-1}_0 (1))$ and $\gamma_cD(F_0)(m_c)=\|D(F_0)\|_\infty\ne 0$. (Such functions and a point $(x_c,y_c,m_c,\gamma_c)$ exist by Lemma \ref{1}.)
 Then for any $0<\theta<\pi/2$, 
$c_\theta=(x_c,y_c,m_c,e^{i\theta}\gamma_c)$ is also in the Choquet boundary for $I(\wbt)$. 
\end{lemma}
\begin{proof}
Let $\theta$ be $0<\theta<\pi/2$. The point evaluation $\phi_\theta(I(F))=F(x_c,y_c)+e^{i\theta}\gamma_cD(F)(m_c)$ at $c_\theta$ is well defined for $I(F)\in I(\wbt)$ since $I$ is injective. We prove that the point evaluation $\phi_\theta$  is an extreme point of the closed unit ball $I(\wbt)^*_1$ of the dual space $I(\wbt)^*$ of $I(\wbt)$. Suppose that $\phi_\theta=\frac12(\phi_1+\phi_2)$ for $\phi_1,\phi_2\in I(\wbt)^*$ with $\|\phi_1\|=\|\phi_2\|=1$, where $\|\cdot\|$ denotes the operator norm here. Let $\check{\phi_j}$ be a Hahn-Banach extension of $\phi_j$ to $C(X_2\times Y_2\times\mathfrak{M}_2\times \mathbb{T})$ for each $j=\theta, 1,2$. By the Riesz-Markov-Kakutani representation theorem there exists a complex regular Borel measure $\mu_j$ on $X_2\times Y_2\times\mathfrak{M}_2\times \mathbb{T}$ with $\|\mu_j\|=1$ which represents $\check{\phi_j}$ for $j=\theta,1,2$ respectively. In particular, we have
\[
\int I(F)d\mu_j=\phi_j(I(F)),\qquad I(F)\in I(\wbt)
\]
for $j=\theta,1,2$. As $\int 1d\mu_\theta=\phi_\theta(1)=1$ we see that $\mu_\theta$ is a probability measure. By the equation 
\[
1=\int 1d\mu_\theta=\frac12\int 1d\mu_1+\frac12\int 1d\mu_2
\]
we see that $\mu_1$ and $\mu_2$ are also probability measures.

We prove that the support $\suppmj$ of the measure  $\mu_j$ satisfies
\begin{equation}\label{sppmj}
\suppmj\subset b^{-1}_0(1)\times f^{-1}_0(1)\times\left\{(K_1\times \{e^{i\theta}\gamma_c\})\cup (K_2\times \{-e^{i\theta}\gamma_c\})\right\},
\end{equation}
where $K_1=D(F_0)^{-1}(D(F_0)(m_c))$ and $K_2=D(F_0)^{-1}(-D(F_0)(m_c))$, for $j=\theta,1,2$. Note that $m_c\in K_1$ while $K_2$ may be empty. Note also that $K_1\cap K_2=\emptyset$ since $|D(F_0)(m_c)|=\|D(F_0)\|_\infty\ne 0$. We first consider the case for $j=\theta$. As $(x_c, y_c)\in b^{-1}_0(1)\times f^{-1}_0(1)$ we have
\[
\phi_\theta(I(F_0))=F_0(x_c,y_c)+e^{i\theta}\gamma_cD(F_0)(m_c)
=1+ e^{i\theta}\gamma_cD(F_0)(m_c).
\]
As $\phi_\theta(I(F_0))=\int I(F_0)d\mu_\theta$ we have
\begin{multline*}
1+ e^{i\theta}\gamma_cD(F_0)(m_c)\\
=\int F_0(x,y)d\mu_\theta(x,y,m,\gamma) + \int \gamma D(F_0)(m)d\mu_\theta(x,y,m,\gamma).
\end{multline*}
Note that $0\le \int F_0(x,y)d\mu_\theta \le 1$ since $0\le F_0\le 1$ and $\mu_\theta$ is a probability measure. As $\gamma_cD(F_0)(m_c)=\|D(F_0)\|_\infty$, we have 
\[
\left|\int\gamma D(F_0)(m)d\mu_\theta\right|\le \gamma_cD(F_0)(m_c).
\]
Taking into account that $0<\theta<\pi/2$ we have by an elementary calculation that 
\begin{equation}\label{(a)}
1=\int F_0(x,y)d\mu_\theta,
\end{equation}
\begin{equation}\label{(b)}
e^{i\theta}\gamma_cD(F_0)(m_c)
=
\int \gamma D(F_0)(m)d\mu_\theta.
\end{equation}
Since $\mu_\theta$ is a regular Borel measure, $\mu_\theta(L)=0$ for any Borel set $L$ with $L\cap \suppmthe=\emptyset$. Hence we have $\int Gd\mu_\theta=\int_{\suppmthe} Gd\mu_\theta$ for every $G\in C(X_2\times Y_2\times \mathfrak{M}_2\times \mathbb{T})$. Then by the equality \eqref{(a)} we have 
\[
1=\int_{\suppmthe} F_0(x,y)d\mu_\theta.
\]
As $0\le F_0\le 1$ we have by Lemma \ref{0.1} that 
\begin{equation}\label{guruguru}
\suppmthe \subset F_0^{-1}(1)\times\mathfrak{M}_2\times\mathbb{T}
=
b^{-1}_0(1)\times f^{-1}_0(1)\times\mathfrak{M}_2\times\mathbb{T}.
\end{equation}
Letting $K=X_2\times Y_2\times \mathfrak{M}$, $g=1_{C(X_2\times Y_2)}\otimes D(F_0)$, and 
applying Lemma \ref{0.3} to the equation \eqref{(b)} we get
\begin{multline*}
\suppmthe \subset 
\left\{X_2\times Y_2\times K_1\times\{e^{i\theta}\gamma_c\}\right\}\cup\left\{ X_2\times Y_2\times K_2\times\{-e^{i\theta}\gamma_c\}\right\} \\
=
X_2\times Y_2\times\left\{(K_1\times\{e^{i\theta}\gamma_c\})\cup (K_2\times \{-e^{i\theta}\gamma_c\})\right\}.
\end{multline*}
Combining this inclusion with \eqref{guruguru} we infer that the inclusion \eqref{sppmj} holds
for $\mu_\theta$. In order to prove the corresponding inclusion for $\mu_j$ for $j=1,2$, we first have
\begin{multline*}
1+e^{i\theta}\gamma_c D(F_0)(m_c)
=
\phi_\theta (I(F_0))\\
=
\int I(F_0)d\frac{\mu_1+\mu_2}{2}+\int \gamma D(F_0)d\frac{\mu_1+\mu_2}{2}
\end{multline*}
by the equation $\phi_\theta(I(F_0))=\frac12\left(\phi_1(I(F_0))+\phi_2(I(F_0))\right)$.
Using a similar argument to that of 
$\mu_\theta$ for $\frac{\mu_1+\mu_2}{2}$ we get
\[
\suppm(\frac{\mu_1+\mu_2}{2}) \subset b^{-1}_0(1)\times f^{-1}_0(1)\times\left\{(K_1\times\{e^{i\theta}\gamma_c\})\cup (K_2\times \{-e^{i\theta}\gamma_c\})\right\}.
\]
As $\mu_1$ and $\mu_2$ are positive measures we have the inclusion \eqref{sppmj} for $j=1,2$.

Next we prove equations
\begin{equation}\label{star1}
F(x_c,y_c)=\int F(x,y)d\mu_\theta
\end{equation}
and
\begin{multline}\label{star2}
D(F)(m_c)=(e^{i\theta}\gamma_c)^{-1}\int \gamma D(F)(m)d\mu_\theta
\\
=
\int_{L_1}D(F)(m)d\mu_\theta
-
\int_{L_2}D(F)(m)d\mu_\theta
\end{multline}
for every $F\in \wbt$, where $L_j=b^{-1}_0(1)\times f^{-1}_0(1)\times K_j\times \{(-1)^{j+1}e^{i\theta}\gamma_c\}$ for $j=1,2$. 
We first show \eqref{star1} and \eqref{star2} for
 a real-valued function $F\in \wbt$. 
Suppose that $F\in \wbt\cap C_{\mathbb R}(X_2\times Y_2)$. Then we have
\begin{equation}\label{grgr}
\begin{split}
F(x_c,y_c)&+e^{i\theta}\gamma_c D(F)(m_c)
 = \phi_\theta(I(F)) \\
& = \int F(x,y)d\mu_\theta + \int \gamma D(F)(m)d\mu_\theta \\
& = \int F(x,y)d\mu_\theta + \int_{L_1} \gamma D(F)(m)d\mu_\theta +\int_{L_2} \gamma D(F)(m)d\mu_\theta \\
& = \int F(x,y)d\mu_\theta  \\
& \qquad
+e^{i\theta}\gamma_c \left(
\int_{L_1}D(F)(m)d\mu_\theta -
\int_{L_2}D(F)(m)d\mu_\theta \right).
\end{split}
\end{equation}
Note that $F(x_c,y_c)$, $D(F)(m_c)$, $\int F(x,y)d\mu_\theta$, $\int_{L_j}D(F)(m)d\mu_\theta$ 
for $j=1,2$ are all real numbers since $F$ and $D(F)$ are real-valued functions (see Definition \ref{aqL}). We also note that $e^{i\theta}\gamma_c\not\in \mathbb{R}$ since $0<\theta<\pi/2$ and $\gamma_c=1$ or $-1$. 
Then comparing the real and the imaginary parts of the equation \eqref{grgr} we have \eqref{star1} and \eqref{star2} 
for every $F\in \wbt \cap C_{\mathbb{R}}(X_2\times Y_2)$. 
Take a general function 
 $F\in \wbt$. We have assumed that $\wbt$ is self-adjoint by the condition $\cno$ in Definition \ref{aqL}, therefore the real part $\operatorname{Re}F$ and the imaginary part $\operatorname{Im}F$ of $F$ both are in $\wbt \cap C_{\mathbb{R}}(X_2\times Y_2)$. Then 
by \eqref{star1}  for real-valued maps,
 we have
\[
\operatorname{Re}F(x_c,y_c)=\int\operatorname{Re}F(x,y)d\mu_\theta,
\]
\[
\operatorname{Im}F(x_c,y_c)=\int\operatorname{Im}F(x,y)d\mu_\theta.
\]
Hence we have
\[
F(x_c,y_c)=\int\operatorname{Re}F(x,y)d\mu_\theta+i\int\operatorname{Im}F(x,y)d\mu_\theta=\int F(x,y)d\mu_\theta.
\]
Thus \eqref{star1} is proved for every $F\in \wbt$. As $D$ is complex-linear we have by \eqref{star2} for real-valued functions that
\begin{multline*}
D(F)(m_c)=D(\operatorname{Re}F)(m_c)+iD(\operatorname{Im}F)(m_c) \\
=(e^{i\theta}\gamma_c)^{-1}\int \gamma D(\operatorname{Re}F)(m)d\mu_\theta
+i(e^{i\theta}\gamma_c)^{-1}\int\gamma D(\operatorname{Im}F)(m)d\mu_\theta \\
= (e^{i\theta}\gamma_c)^{-1}\int \gamma D(F)d\mu_\theta \\
=
\int_{L_1}D(F)(m)d\mu_\theta
-
\int_{L_2}D(F)(m)d\mu_\theta.
\end{multline*}
Thus we have just proved  \eqref{star2} for every $F\in \wbt$. 

For every $F\in \wbt$ we have
\begin{equation*}
\begin{split}
\phi_\theta (I(F)) &=\frac12 \left(\phi_1(I(F))+\phi_2(I(F))\right) \\
&=
\int F(x,y)d\frac{\mu_1+\mu_2}{2}+\int \gamma D(F)(m)d\frac{\mu_1+\mu_2}{2}.
\end{split}
\end{equation*}
By the same way as the proof of \eqref{star1} and \eqref{star2} we have 
\begin{equation}\label{star3}
F(x_c,y_c)=\int F(x,y)d\frac{\mu_1+\mu_2}{2}
\end{equation}
and
\begin{multline}\label{star4}
D(F)(m_c)=(e^{i\theta}\gamma_c)^{-1}\int \gamma D(F)(m)d\frac{\mu_1+\mu_2}{2}
\\
=
\int_{L_1}D(F)(m)d\frac{\mu_1+\mu_2}{2}
-
\int_{L_2}D(F)(m)d\frac{\mu_1+\mu_2}{2}
\end{multline}
for every $F\in \wbt$. 

Next define a regular Borel probability measure $\nu_j$ on $X_2\times Y_2\times \mathfrak{M}_2\times \mathbb{T}$ for $j=\theta,1,2$ by
\[
\nu_j(E)=\mu_j(\{(x,y,m,e^{i\theta}\gamma):(x,y,m,\gamma)\in E\})
\]
for a Borel set $E\subset X_2\times Y_2\times\mathfrak{M}_2\times\mathbb{T}$. Then we have 
\begin{equation}\label{A1}
\int F(x,y)d\nu_j=\int F(x,y) d\mu_j
\end{equation}
for every $F\in \wbt$ and $j=\theta, 1,2$. By \eqref{sppmj}
 we have
\begin{equation}\label{sppnj}
\operatorname{supp}(\nu_j)\subset b^{-1}_0(1)\times f^{-1}_0(1)\times \left[(K_1\times \{\gamma_c\})\cup (K_2\times \{-\gamma_c\})\right]
\end{equation}
for $j=\theta,1,2$. Put $T_j=b^{-1}_0(1)\times f^{-1}_0(1)\times K_j\times\{(-1)^{j+1}\gamma_c\}$. As $\nu_\theta$ and $\frac{\nu_1+\nu_2}{2}$ are regular and $K_1\cap K_2=\emptyset$, we have  by \eqref{sppmj} and \eqref{sppnj} that 
\begin{equation}\label{A2}
\begin{split}
\int\gamma D(F)(m)d\nu_j 
& = \int_{T_1}\gamma D(F)(m)d\nu_j+\int_{T_2}\gamma D(F)(m)d\nu_j \\
& = \gamma_c\int_{T_1}D(F)(m)d\nu_j-\gamma_c\int_{T_2}D(F)(m)d\nu_j \\
& = \gamma_c\int_{L_1}D(F)(m)d\mu_j-\gamma_c\int_{L_2}D(F)(m)d\mu_j \\
& = e^{-i\theta}\int \gamma D(F)(m)d\mu_j
\end{split}
\end{equation}
for every $F\in \wbt$ and $j=\theta, 1,2$. 
For $j=\theta, 1,2$, put $\psi_j:I(\wbt)\to \mathbb{C}$ 
by 
\[
\psi_j(I(F))=\int I(F)d\nu_j, \quad I(F)\in I(\wbt).
\]
As $\nu_j$ is a probability measure we see that $\psi_j\in I(\wbt)^*_1$. Let $I(F)\in I(\wbt)$. Then by \eqref{A1} and \eqref{A2} we have
\begin{equation*}
\begin{split}
\psi_\theta (I(F)) 
& = \int I(F)d\nu_\theta \\
& = \int F(x,y)d\nu_\theta +\int \gamma D(F)(m)d\nu_\theta \\
& = \int F(x,y)d\mu_\theta +e^{-i\theta}\int \gamma D(F)(m)d\mu_\theta.
\end{split}
\end{equation*}
Then by \eqref{star1} and \eqref{star2} we have 
\[
\psi_\theta(I(F))=F(x_c,y_c)+\gamma_cD(F)(m_c)=I(F)(x_c,y_c,m_c,\gamma_c).
\]
That is $\psi_\theta$ is the point evaluation for $I(\wbt)$ at $(x_c,y_c,m_c,\gamma_c)$. 
By \eqref{A1}, \eqref{A2}, \eqref{star3} and \eqref{star4} we have
\begin{multline*}
\frac12(\psi_1(I(F))+\psi_2(I(F))) \\
= \int F(x,y)d\frac{\nu_1+\nu_2}{2}+\int \gamma D(F)(m)d\frac{\nu_1+\nu_2}{2} \\
= \int F(x,y)d\frac{\mu_1+\mu_2}{2}+e^{-i\theta}\int \gamma D(F)(m)d\frac{\mu_1+\mu_2}{2} \\
= F(x_c,y_c)+\gamma_c D(F)(m_c)
\end{multline*}
for every $F\in \wbt$. Hence we have 
\[
\psi_\theta(I(F))=\frac12\left(\psi_1(I(F))+\psi_2(I(F))\right)
\]
for every $I(F)\in I(\wbt)$; $\psi_\theta=\frac12(\psi_1+\psi_2)$. Since $(x_c,y_c,m_c,\gamma_c)$ is in the Choquet boundary for $I(\wbt)$, $\psi_\theta$ is an extreme point for $I(\wbt)^*_1$. Thus we have that $\psi_\theta=\psi_1=\psi_2$. 

Applying the equations $\psi_\theta=\psi_1=\psi_2$ we prove that $\phi_\theta=\phi_1=\phi_2$. By \eqref{A1} and \eqref{A2} we have
\begin{multline}\label{K}
\phi_j(I(F))=\int F(x,y)d\mu_j+\int \gamma D(F)(m)d\mu_j \\
=\int F(x,y)d\nu_j+ e^{i\theta}\int \gamma D(F)(m)d\nu_j,\quad F\in I(\wbt)
\end{multline}
for every $j=\theta,1,2$. 
Put
\[
P=\{G\in \wbt: 0\le G\le 1=G(x_c,y_c)\}.
\]
Then the set $P$ separates the points of $X_2\times Y_2$. Suppose that $(x_1,y_1)$ and $(x_2,y_2)$ are different points in $X_2\times Y_2$. We may assume that $(x_c,y_c)\ne(x_2,y_2)$. Let $\mathfrak{U}_c$ be an open neighbourhood of $(x_c,y_c)$ such that $(x_2,y_2)\not\in \mathfrak{U}_c$. By Lemma \ref{1} there is $F_c\in \wbt$ such that $0\le F_c\le 1=F_c(x_c,y_c)$ on $X_2\times Y_2$ and $F_c<1/2$ on $X_2\times Y_2\setminus \mathfrak{U}_c$. Hence $0\le F_c(x_2,y_2)<1/2$. In the same way there exists $F_1\in \wbt$ such that $0\le F_1\le 1=F_1(x_1,y_1)$ on $X_2\times Y_2$ and $0\le F_1(x_2,y_2)<1/2$. Put $H=1-(1-F_c)(1-F_1)\in \wbt$. Then we infer that $0\le H\le 1$ on $X_2\times Y_2$, $H(x_c,y_c)=H(x_1,y_1)=1$, and $H(x_2,y_2)\ne 1$. Hence we have that $H\in P$ and $H(x_1,y_1)\ne H(x_2,y_2)$.
Let $G\in P$ be arbitrary. Since $P\subset \wbt$, we have $G\in \wbt$. Hence by the equality \eqref{A1} we have
\begin{multline*}
\frac12\left(\int G(x,y)d\nu_1+\int G(x,y)d\nu_2\right) \\
=\frac12\left(\int G(x,y)d\mu_1+\int G(x,y)d\mu_2\right)
=\int G(x,y)d\frac{\mu_1+\mu_2}{2}.
\end{multline*}
By \eqref{star3}
\[
\int G(x,y)d\frac{\mu_1+\mu_2}{2}=G(x_c,y_c)=1.
\]
Hence we have 
\[
\frac12\left(\int G(x,y)d\nu_1+\int G(x,y)d\nu_2\right)=1.
\]
Since $0\le G\le 1$ we have $0\le \int G(x,y)d\nu_j\le 1$ for $j=1,2$. It follows that 
\[
\int G(x,y)d\nu_1=\int G(x,y)d\nu_2 = 1.
\]
As $G\in P$ is arbitrary we have 
\[
\int \sum a_nG_n(x,y)d\nu_1=\sum a_n= \int \sum a_nG_n(x,y)d\nu_2
\]
for any complex linear combination $\sum a_nG_n$ for 
$G_n\in P$. 
Since $P$ is closed under multiplication and separates the points in $X_2\times Y_2$, we have that
\[
\left\{\sum a_nG_n:\text{$a_n\in {\mathbb C}$, $G_n\in P$}\right\}
\] 
is a unital subalgebra of $\wbj$ which is conjugate-closed and separates the points of $X_2\times Y_2$. The Stone-Weierstrass theorem asserts that it is uniformly dense in $C(X_2\times Y_2)$, hence so is in $\wbt$. It follows that we have
\begin{equation}\label{C}
\int F(x,y)d\nu_1=\int F(x,y)d\nu_2
\end{equation}
for every $F\in \wbt$. On the other hand, since $\psi_1=\psi_2$ we have
\begin{multline}\label{nnn}
\int F(x,y)d\nu_1+\int \gamma D(F)(m)d\nu_1=\psi_1(I(F))\\
=\psi_2(I(F))=
\int F(x,y)d\nu_2+\int \gamma D(F)(m)d\nu_2
\end{multline}
for every $F\in \wbt$. By \eqref{C} and \eqref{nnn} we have
\[
\int \gamma D(F)(m)d\nu_1=\int \gamma D(F)(m) d\nu_2
\]
for every $F\in \wbt$. It follows by \eqref{K} that $\phi_1(I(F))=\phi_2(I(F))$ for every $F\in \wbt$. We infer that $\phi_\theta=\phi_1=\phi_2$. We conclude that $\phi_\theta$ is an extreme point for any $0<\theta<\pi/2$, that is, $(x_c,y_c,m_c,e^{i\theta}\gamma_c)$ is in the Choquet boundary for $I(\wbt)$ for  any $0<\theta<\pi/2$.
\end{proof}
\begin{proof}[Proof of Proposition \ref{absolute value 1}]
Define a map $\tilde{U}:I_1(\wbo)\to I_2(\wbt)$ by $\tilde{U}(I_1(H))=I_2(U(H))$ for $I_1(H)\in I_1(\wbo)$. The map $\tilde{U}$ is well defined since $I_1$ is injective. Due to the definition of $I_j$, we see that $\tilde{U}$ is a surjective isometry.
Then the dual map $\tilde{U}^*:I_2(\wbt)^*\to I_1(\wbo)^*$ is an isometry and it preserves the extreme points of the closed unit ball $I_2(\wbt)^*_1$ of $I_2(\wbt)^*$. 
Let $(x_0,y_0)$ be an arbitrary point in $X_2\times Y_2$ and ${\mathfrak U}$ an arbitrary open neighborhood of $(x_0,y_0)$. Then by Lemmata \ref{1} and \ref{2} there exists $(x_c,y_c,m_c,\gamma_c)\in {\mathfrak U}\times\mathfrak{M}_2\times \mathbb{T}$ such that $(x_c,y_c,m_c,e^{i\theta}\gamma_c)$ is in the Choquet boundary of $I(\wbt)$ for every $0\le \theta<\pi/2$. Let $\phi_\theta$ be the point evaluation on $I(\wbt)$ at $(x_c,y_c,m_c,e^{i\theta}\gamma_c)$. Then $\phi_\theta$ is an extreme point of the closed unit ball $I(\wbt)^*_1$. As $\tilde{U}^*$ preserves the extreme point of the closed unit ball, 
$\tilde{U}^*(\phi_{\theta})$ is an extreme points of the closed unit ball $I_1(\wbo)^*_1$ of $I_1(\wbo)^*$. By the Arens-Kelly theorem we see that there exists a complex number $\gamma$ with absolute value 1 and a point $d$ in the Choquet boundary for $I_1(\wbo)$ such that $\tilde{U}^*(\phi_{\theta})=\gamma \phi_d$, where $\phi_d$ denotes the point evaluation for $I_1(\wbo)$ at $d$. Thus we have that
\[
|\tilde{U}^*(\phi_{\theta})(1)|=1.
\]
As $\tilde{U}^*(\phi_\theta)(1)=\phi_\theta(I_2(U(1)))$ we have
\[
1=|U(1)(x_c,y_c)+e^{i\theta}\gamma_cD(U(1))(m_c)|
\]
for every $0\le \theta <\pi/2$. Hence one of the following (i) or (ii) occurs:
\begin{itemize}
\item[(i)]
$U(1)(x_c,y_c)=0$ and $|D(U(1))(m_c)|=1$,
\item[(ii)]
$|U(1)(x_c,y_c)|=1$ and $D(U(1))(m_c)=0$.
\end{itemize}
But (i) never occurs. The reason is as follows. Since $U$ is an isometry we have 
\begin{equation}\label{E}
1=\|1\|=\|U(1)\|=\|U(1)\|_\infty +\|D(U(1))\|_\infty.
\end{equation}
Suppose that (i) holds. 
By the second equation of (i) we have $\|D(U(1))\|_\infty\ge 1$. Then by \eqref{E} we have $\|U(1)\|_{\infty}=0$, and $U(1)=0$, which contradicts \eqref{E}. Thus we conclude that only (ii) occurs. 

By the first equation of (ii) we infer that $\|U(1)\|_\infty\ge 1$. Then by the equation \eqref{E}, we have $0=\|D(U(1))\|_\infty$. By the condition $\cnt (2)$  of Definition \ref{aqL} we have $U(1)\in 1\otimes C(Y_2)$; there exists $h\in C(Y_2)$ with $U(1)=1\otimes h$. As $|U(1)(x_c,y_c)|=1$ we have $|h(y_c)|=1$. Note that $h$ does not depend on the point $(x_0,y_0)$ nor a neighborhood $\mathfrak{U}$. As $\mathfrak{U}$ is an arbitrary neighborhood of $(x_0,y_0)$, and $(x_c,y_c)\in \mathfrak{U}$, the continuity of $h$ asserts that $|h(y_0)|=1$. Since $y_0$ is an arbitrary point in $Y_2$, we infer that $|h|=1$ on $Y_2$.
\end{proof}
\section{Proof of Theorem \ref{main}}
\begin{proof}[Proof of Theorem \ref{main}]
Suppose first $X_1=\{x_1\}$ and $X_2=\{x_2\}$ are singletons. In this case $B_j$ is isometrically isomorphic to $\mathbb{C}$ as a Banach algebra and $\wbj=1\otimes C(Y_j)$. Thus $\|D(F)\|_\infty=0$ for every $F\in \wbj$. Therefore $\wbj$ is isometrically isomorphic to $C(Y_j)$ for $j=1,2$. Thus we may suppose that $U$ is a surjective isometry from $C(Y_1)$ onto $C(Y_2)$. Then applying the Banach--Stone theorem, we see that $|U(1)|=1$ on $Y_2$ and there exists a homeomorphism $\tau:Y_2\to Y_1$ such that
\[
U(F)=U(1)F\circ \tau, \qquad F\in C(Y_1).
\]
Letting $U(1)=1\otimes h$ and $\varphi:X_2\times Y_2 \to X_1$ by $\varphi(x_2,y)=x_1$ for every $y\in Y_2$, we have 
\[
U(F)(x,y)=h(y)F(\varphi(x,y),\tau(y)),\qquad (x,y)\in X_2\times Y_2
\]
for every $F\in \wbo$.

Suppose that $X_2$ is not a singleton. We prove the conclusion applying Proposition \ref{absolute value 1}. By Proposition \ref{absolute value 1} there exists $h\in C(Y_2)$ with $|h|=1$ on $Y_2$ such that $U(1)=1\otimes h$. Define $U_0:\wbo\to \wbt$ by $U_0(F)=1\otimes \bar{h}U(F)$ for $F\in \wbo$, where $\bar{h}$ denotes the complex conjugate of $h$. It is easy to see that $U_0$ is a bijection with $U_0(1)=1$. By the condition $\cnt (3)$  of Definition \ref{aqL} it is also easy to check that $U_0$ is an isometry. As $\wbj$ is a unital Banach algebra which is contained in $C(X_j\times Y_j)$ which separates the points of $X_j\times Y_j$. As $\wbj$ is natural, by \cite[Proposition 2]{ja} it is a regular subspace of $C(X_j\times Y_j)$ in the sense of Jarosz \cite[p. 67]{ja}. As the norm $\|\cdot\|=\|\cdot\|_\infty+\nn\cdot\nn$ is a $p$-norm (see \cite[p. 67]{ja}) and $U_0(1)=1$, we infer by Theorem in \cite{ja} that $U_0$ is also an isometry with respect to the supremum norm $\|\cdot\|_{\infty}$ on $X_j\times Y_j$. As $\wbj$ is a self-adjoint unital subalgebra of $C(X_j\times Y_j)$ which separates the points of $X_j\times Y_j$, the Stone-Weierstrass theorem asserts that $\wbj$ is uniformly dense in $C(X_j\times Y_j)$. Then the Banach--Stone theorem asserts that $U_0$ is an algebra isomorphism. Since $U_0$ is an isometry with respect to the original norm $\|\cdot\|$ on $\wbj$ we have for every $1\otimes g\in 1\otimes C(Y_1)$ that
\begin{multline*}
\|1\otimes g\|_{\infty}+\|D(1\otimes g)\|_\infty=\|1\otimes g\|=\|U_0(1\otimes g)\| \\
=\|U_0(1\otimes f)\|_\infty+\|D(U_0(1\otimes g))\|_\infty.
\end{multline*}
By the condition $\cnt (2)$ of Definition \ref{aqL} we have $\|D(1\otimes g)\|_\infty=0$. Since $U_0$ is also an isometry with respect to the supremum norm we have $\|1\otimes g\|_\infty=\|U_0(1\otimes g)\|_\infty$. Therefore we have that $\|D(U_0(1\otimes g))\|_\infty=0$. By the condition $\cnt (2)$ of Definition \ref{aqL}  we have that $U_0(1\otimes g)\in 1\otimes C(Y_2)$. Hence we see that $U_0(1\otimes C(Y_1))\subset 1\otimes C(Y_2)$. By the Stone-Weierstrass theorem $B_1\otimes C(Y_1)$ is uniformly dense in $C(X_1\times Y_1)$, hence $\wbo\subset \overline{B_1\otimes C(Y_1)}$, where $\overline{\cdot}$ denotes the uniform closure on $X_1\times Y_1$. Then by Proposition 3.2 and the comments which follow that proposition in \cite{hots} there exists continuous maps $\varphi:X_2\times Y_2\to X_1$ and  $\tau:Y_2\to Y_1$ such that 
\begin{equation}\label{abc}
U_0(F)(x,y)=F(\varphi(x,y),\tau(y)),\qquad (x,y)\in X_2\times Y_2
\end{equation}
for every $F\in \wbo$. 
As $X_2$ is not a singleton, there are two distinct points $z,w\in X_2$. Let $y\in Y_2$ be any point. As $U_0$ is a surjecton and $\wbt$ separates the points of $X_2\times Y_2$, there exists a function $F\in \wbo$ such that $U_0(F)(z,y)\ne U_0(F)(w,y)$. Then by \eqref{abc} we have
\[
F(\varphi(z,y),\tau(y))=U_0(F)(z,y)\ne U_0(F)(w,y)=F(\varphi(w,y),\tau(y)).
\]
Hence $\varphi(z,y)\ne \varphi(w,y)$. As $\varphi(z,y),\varphi(w,y)\in X_1$, we have that $X_1$ is not a singleton. 
Applying a similar argument for $U_0^{-1}$ instead of $U_0$ we observe that there exists continuous maps $\varphi_1:X_1\times Y_1\to X_2$ and $\tau_1:Y_1\to Y_2$ such that 
\[
U_0^{-1}(G)(u,v)=G(\varphi_1(u,v), \tau_1(v)),\qquad (u,v)\in X_1\times Y_1
\]
for every $G\in \wbt$. Thus we have 
\begin{multline}\label{[1]}
G(x,y)=U_0(U_0^{-1}(G))(x,y)=U_0^{-1}(G)(\varphi(x,y),\tau(y)) \\
=G(\varphi_1(\varphi(x,y),\tau(y)), \tau_1(\tau(y))), \quad (x,y)\in X_2\times Y_2
\end{multline}
for every $G\in \wbt$ and 
\begin{multline}\label{[2]}
F(u,v)=U_0^{-1}(U_0(F))(u,v)=U_0(F)(\varphi_1(u,v),\tau_1(v)) \\
=F(\varphi(\varphi_1(u,v),\tau_1(v)),\tau(\tau_1(v))), \quad (u,v)\in X_1\times Y_1
\end{multline}
for every $F\in \wbo$. As $\wbo$ separates the points in $X_1\times Y_1$ and $\wbt$ separates the points in $X_2\times Y_2$, we infer that $y=\tau_1(\tau(y))$ for every $y\in Y_2$ and $v=\tau(\tau_1(v))$ for every $v\in Y_1$. Hence $\tau:Y_2\to Y_1$ and $\tau_1:Y_1\to Y_2$ are homeomorphisms and $\tau_1^{-1}=\tau$. We have by \eqref{[2]} that $u=\varphi(\varphi_1(u,v),\tau_1(v))$ for every $(u,v)\in X_1\times Y_1$. As $\tau_1$ is a homeomorphism, we infer that $u=\varphi(\varphi_1(u,\tau_1^{-1}(y)),y)$ holds for every pair $u\in X_1$ and $y\in Y_2$. It means that for every $y\in Y_2$ the map $\varphi(\cdot,y):X_2\to X_1$ is a surjection. 

We prove that $\varphi(\cdot,y)$ is an injection for every $y\in Y_2$. Let $y\in Y_2$. Suppose that $\varphi(a,y)=\varphi(b,y)$ for $a,b\in X_2$. Then $\varphi_1(\varphi(a,y),\tau(y))=a$ and $\varphi_1(\varphi(b,y),\tau(y))=b$ by the equation \eqref{[1]}. Thus we have $a=b$. Hence we conclude that $\varphi(\cdot,y)$ is an injection. It follows that $\varphi(\cdot,y):X_2\to X_1$ is a bijective continuous map. As $X_2$ is compact and $X_1$ is Hausdorff, we at once see that $\varphi(\cdot,y)$ is a homeomorphism. As $U_0(F)=1\otimes\bar{h} U(F)$ for every $F\in \wbo$ we conclude that 
\[
U(F)(x,y)=h(y)F(\varphi(x,y),\tau(y)),\qquad (x,y)\in X_2\times Y_2.
\]

Suppose that $X_1$ is not a singleton. By a similar argument for $U^{-1}$ instead of $U$ we see that there exists a continuous map $\varphi_1:X_1\times Y_1\to X_2$ such that $\varphi_1(\cdot,y):X_1\to X_2$ is a homeomorphism. As $X_1$ is not a singleton we infer that $X_2$ is not a singleton. Then the conclusion follows from the proof for the case where $X_2$ is not a singleton.
\end{proof}
\section{Examples of admissible quadruples of type L with applications of Theorem \ref{main}}\label{example}
\begin{example}\label{lip}
Let $(X,d)$ be a compact metric space and $Y$ a compact Hausdorff space. Let $0<\alpha\le 1$. Suppose that $B$ is a closed subalgebra of $\Lip((X,d^\alpha))$ which contains the constants and separates the points of $X$, where $d^\alpha$ is the H\"older metric induced by $d$. Suppose that $\widetilde{B}$ is a closed subalgebra of $\Lip((X,d^\alpha),C(Y))$ which contains the constants and separates the points of $X\times Y$. Suppose that $B$ and $\wb$ are self-adjoint. Suppose that 
\[
B\otimes C(Y)\subset \widetilde{B}
\]
 and 
\[
\{F(\cdot, y):F\in \widetilde{B}, y\in Y\}\subset B.
\] 
Let $\mathfrak{M}$ be the Stone-\v Cech compactification of $\{(x,x')\in X^2:x\ne x'\}\times Y$. For $F\in \wb$, let $D(F)$ be the continuous extension to $\mathfrak{M}$ of the function $(F(x,y)-F(x',y))/d^\alpha(x,x')$ on $\{(x,x')\in X^2:x\ne x'\}\times Y$. Then $D:\wb\to C(\mathfrak{M})$ is well defined. 
We have $\|D(F)\|_\infty=L_\alpha(F)$ for every $F\in \wb$.
It is easy to see that the condition $\cnt$ of Definition \ref{aqL} is satisfied. Hence we have that 
 $(X,C(Y),B,\widetilde{B})$ is an admissible quadruple of type L. 

There are two typical example of $(X,C(Y),B,\widetilde{B})$ above. One is
\[
(X,C(Y), \Lip((X,d^\alpha)),\Lip((X,d^\alpha),C(Y)))
\]
By Corollary \ref{eell} $\Lip((X,d^\alpha))$ and $\Lip((X,d^\alpha),C(Y))$ are self-adjoint. The inclusions 
\[
\Lip((X,d^\alpha))\otimes C(Y)\subset \Lip((X,d^\alpha),C(Y))
\]
and 
\[
\{F(\cdot,y):F\in \Lip((X,d^\alpha),C(Y)), y\in Y\}\subset \Lip((X,d^\alpha))
\]
 is obvious. The other example of $(X,C(Y),B,\widetilde{B})$ above is
\[
(X, C(Y), \lip(X), \lip(X,C(Y)))
\]
for $0<\alpha <1$. 
 In fact $\lip(X)$ (resp. $\lip(X,C(Y))$) is a closed subalgebra of $\Lip((X,d^\alpha))$ (resp. $\Lip((X,d^\alpha),C(Y))$ which contains the constants. In this case   Corollary \ref{eell} asserts that $\lip(X)$ separates the points of $X$. As $\lip(X)\otimes C(Y)\subset \lip(X,C(Y))$ we see that $\wb=\lip(X,C(Y))$ separates the points of $X\times Y$. By Corollary \ref{eell}  $\lip(X)$ and $\lip(X,C(Y))$ are self-adjoint. The inclusions 
\[
\lip(X)\otimes C(Y)\subset \lip(X,C(Y))
\]
and 
\[
\{F(\cdot,y):F\in \lip(X,C(Y)), y\in Y\}\subset \lip(X)
\]
is obvious.
\end{example}
\begin{cor}\label{g}
Let $j=1,2$.
Let $(X_j,d_j)$ be a compact metric space 
and $Y_j$ a compact Hausdorff space. Let $\alpha$ be $0<\alpha\le 1$. Suppose that $B_j$  is a closed subalgebra of $\Lip((X_j,d_j^\alpha))$ which contains the constants and separates the points of $X_j$. Suppose that $\widetilde{B_j}$ is a closed subalgebra of $\Lip((X_j,d_j^\alpha),C(Y_j))$ which contains the constants and separates the points of $X_j\times Y_j$. 
Suppose that $B_j$ and $\wb_j$ are self-adjoint. Suppose that 
\[
B_j\otimes C(Y_j)\subset \widetilde{B_j}
\]
 and 
\[
\{F(\cdot, y):F\in \widetilde{B_j}, y\in Y_j\}\subset B_j.
\] 
Suppose that 
\[
U:\widetilde{B_1}\to \wbt
\]
is a surjective isometry. Then there exists $h\in C(Y_2)$ such that $|h|=1$ on $Y_2$, a continuous map $\varphi:X_2\times Y_2\to X_1$ such that $\varphi(\cdot,y):X_2\to X_1$ is a homeomorphism for each $y\in Y_2$, and a homeomorphism $\tau:Y_2\to Y_1$ which satisfy
\[
U(F)(x,y)=h(y)F(\varphi(x,y),\tau(y)),\qquad (x,y)\in X_2\times Y_2
\]
for every $F\in \wbo$.
\end{cor}
\begin{proof}
As in a similar way to the argument in Example \ref{lip} we see that $(X_j,C(Y_j), B_j, \widetilde{B_j})$ is an admissible quadruple of type L. Then applying Theorem \ref{main} the conclusion holds.
\end{proof}
Note that Corollary \ref{g} holds for $\wbj=\Lip(X_j,C(Y_j))$ and $\wbj=\lip(X_j,C(Y_j))$ for $0<\alpha<1$. 
In this case we have a complete description of a surjective isometry for $\wbj=\Lip(X_j,C(Y_j))$ and $\wbj=\lip(X_j,C(Y_j))$ for $0<\alpha<1$. Note that $\Lip_\alpha((X_j,d_j),C(Y_j))$ for $0<\alpha< 1$ is isometrically isomorphic to $\Lip((X_j,d_j^\alpha),C(Y_j))$ by considering  the  H\"older metric $d_j(\cdot,\cdot)^\alpha$ for the original metric $d_j(\cdot,\cdot)$ on $X_j$.
\begin{cor}\label{isoLip}
Let $(X_j,d_j)$ be a compact metric space and $Y_j$ a compact Hausdorff space for $j=1,2$.
Suppose that 
$U:\Lip(X_1,C(Y_1))\to \Lip(X_2,C(Y_2))$ {\rm (}resp. $U:\lip(X_1,C(Y_1))\to \lip(X_2,C(Y_2))${\rm )}
 is a map. Then $U$ is a surjective isometry with respect to the sum norm $\|\cdot\|=\|\cdot\|_\infty+L(\cdot)$ {\rm (}resp. $\|\cdot\|=\|\cdot\|_\infty+L_\alpha(\cdot)${\rm )} if and only if there exists $h\in C(Y_2)$ with $|h|=1$ on $Y_2$, a continuous map $\varphi:X_2\times Y_2\to X_1$ such that $\varphi(\cdot,y):X_2\to X_1$ is a surjective isometry for every $y\in Y_2$, and a homeomorphism $\tau:Y_2\to Y_1$ which satisfy that
\[
U(F)(x,y)=h(y)F(\varphi(x,y),\tau(y)),\qquad (x,y)\in X_2\times Y_2
\]
for every $F\in \Lip(X_1,C(Y_1))$ {\rm (}resp. $F\in \lip(X_1,C(Y_1))${\rm )}.
\end{cor}
\begin{proof}
Suppose that there exists $h\in C(Y_2)$ with $|h|=1$ on $Y_2$, a continuous map $\varphi:X_2\times Y_2\to X_1$ such that $\varphi(\cdot,y):X_2\to X_1$ is a surjective isometry for every $y\in Y_2$, and a homeomorphism $\tau:Y_2\to Y_1$ which satisfy that
\[
U(F)(x,y)=h(y)F(\varphi(x,y),\tau(y)),\qquad (x,y)\in X_2\times Y_2
\]
for every $F\in \Lip(X_1,C(Y_1))$ {\rm (}resp. $F\in \lip(X_1,C(Y_1))${\rm )}. 
We prove that $U$ is a surjective isometry on $\Lip(X_j,C(Y_j))$. A proof for the case of $\lip(X_j,C(Y_j))$ is the same and we omit it. Since $\varphi(\cdot,y)$ is an isometry for every $y\in Y_2$, we have
\begin{equation}\label{pisoLip5}
\begin{split}
&\frac{|(U(F))(x,y)-(U(F))(x',y)|}{d_2(x,x')}
\\
& = \frac{|h(y)F(\varphi(x,y),\tau(y))-h(y)F(\varphi(x',y),\tau(y))|}{d_2(x,x')} 
\\
& = \frac{|F(\varphi(x,y),\tau(y))-F(\varphi(x',y),\tau(y))|}{d_2(\varphi(x,y),\varphi(x',y))},\quad x,x'\in X_2, y\in Y_2
\end{split}
\end{equation}
for $F\in \Lip(X_1,C(Y_1))$. Since $\varphi(\cdot,y)$ is bijective and the map $(x,y)\mapsto (\varphi(x,y),\tau(y))$ gives a bijection from $X_2\times Y_2$ onto $X_1\times Y_1$, we see by \eqref{pisoLip5} that $L(F)=L(U(F))$ for every $F\in \Lip(X_1,C(Y_1))$. Since $\|F\|_\infty=\|U(F)\|_\infty$, we conclude that
\[
\|F\|=\|F\|_\infty+L(F)=\|U(F)\|_\infty+L(U(F))=\|U(F)\|
\]
for every $F\in \Lip(X_1,C(Y_1))$; that is $U$ is an isometry. We prove that $U$ is surjective. Let $G\in \Lip(X_2,C(Y_2))$ be arbitrary. Put $F$ by 
$F(x,y)=\bar{h}(\tau^{-1}(y))G((\varphi(\cdot,\tau^{-1}(y)))^{-1}(x),\tau^{-1}(y))$ for $(x,y)\in X_1\times Y_1$, where $(\varphi(\cdot,\tau^{-1}(y)))^{-1}$ denotes the inverse of $\varphi(\cdot,\tau^{-1}(y)):X_2\to X_1$. Then we infer that $F\in \Lip(X_1,C(Y_1))$ and $U(F)=G$. As $G$ is an arbitrary elements in $\Lip(X_2,C(Y_2))$, we conclude that $U$ is surjective. It follows that $U$ is a surjective isometry.

Next we prove the converse. First consider the case of $\Lip(X_j,C(Y_j))$. Suppose that $U:\Lip(X_1,C(Y_1))\to \Lip(X_2,C(Y_2))$ is a surjective isometry. Then by Corollary \ref{g} there exists $h\in C(Y_2)$ with $|h|=1$ on $Y_2$, a continuous map $\varphi:X_2\times Y_2\to X_1$ such that $\varphi(\cdot,y):X_2\to X_1$ is a homeomorphism for every $y\in Y_2$, and a homeomorphism $\tau:Y_2\to Y_1$ which satisfy that
\begin{equation}\label{pisoLip1}
U(F)(x,y)=h(y)F(\varphi(x,y),\tau(y)),\qquad (x,y)\in X_2\times Y_2
\end{equation}
for every $F\in \Lip(X_1,C(Y_1))$. We only need to prove that $\varphi(\cdot,y):X_2\to X_1$ is a surjective isometry for every $y\in Y_2$. Let $x_1,x_2\in X_2$ and $y\in Y_2$  be arbitrary. Set $f:X_1\to {\mathbb C}$ by $f(x)=d_1(x,\varphi(x_2,y))$ for $x\in X_1$. Then $f\otimes 1\in \Lip(X_1,C(Y_1))$ and $L(f\otimes 1)=1$. Then we have
\begin{equation}\label{pisoLip2}
\begin{split}
d_1(\varphi(x_1,y),\varphi(x_2,y))
& = f(\varphi(x_1,y))=|f(\varphi(x_1,y))-f(\varphi(x_2,y))|
\\
& = |f\otimes 1(\varphi(x_1,y),\tau(y))-f\otimes 1(\varphi(x_2,y),\tau(y))|
\\
& = |(U(f\otimes 1))(x_1,y)-(U(f\otimes 1))(x_2,y)|
\\
& \le 
L(U(f\otimes 1))d_2(x_1,x_2).
\end{split}
\end{equation}
By \eqref{pisoLip1} the map $U$ is an isometry with respect to $\|\cdot\|_{\infty}$, thus $1=L(f\otimes 1)=L(U(f\otimes 1))$ since $U$ is an isometry for $\|\cdot\|=\|\cdot\|_\infty+L(\cdot)$. It follows by \eqref{pisoLip2} that $d_1(\varphi(x_1,y),\varphi(x_2,y))\le d_2(x_1,x_2)$. Since $U^{-1}$ is a surjective isometry we have by Corollary \ref{g} that there exists $h_1$, $\varphi_1$ and $\tau_1$ such that
\[
U^{-1}(G)(x,y)=h_1(y)G(\varphi_1(x,y),\tau_1(y)),\qquad (x,y)\in X_1\times Y_1
\]
for $G\in \Lip(X_2,C(Y_2))$. Then by a similar way as above we infer that $d_2(\varphi_1(x_1',y'),\varphi_1(x_2',y'))\le d_1(x_1',x_2')$ for every pair $x_1',x_2'\in X_1$ and $y'\in Y_1$. By a simple calculation we obtain that $x=\varphi_1(\varphi(x,y),\tau(y))$ for every $x\in X_2$ and $y\in Y_2$ (see a similar calculation in the proof of Theorem \ref{main} or that given on p.386 of \cite{ho}). Thus we have
\begin{multline*}
d_2(x_1,x_2)=d_2(\varphi_1(\varphi(x_1,y),\tau(y)),\varphi_1(\varphi(x_2,y),\tau(y))\\
\le
d_1(\varphi(x_1,y),\varphi(x_2,y)).
\end{multline*}
Therefore $d_2(x_1,x_2)= d_1(\varphi(x_1,y),\varphi(x_2,y))$ holds for every pair $x_1,x_2\in X_2$ and $y\in Y_2$, that is, $\varphi(\cdot,y)$ is an isometry for every $y\in Y_2$. 

Next we consider the case of $\lip(X_j,C(Y_j))$. Suppose that $0<\alpha<1$ and $U:\lip(X_1,C(Y_1))\to \lip(X_2,C(Y_2))$ is a surjective isometry. As in the same way as before there exists $h\in C(Y_2)$ with $|h|=1$ on $Y_2$, a continuous map $\varphi:X_2\times Y_2\to X_1$ such that $\varphi(\cdot,y):X_2\to X_1$ is a homeomorphism for every $y\in Y_2$, and a homeomorphism $\tau:Y_2\to Y_1$ which satisfy that
\begin{equation*}
U(F)(x,y)=h(y)F(\varphi(x,y),\tau(y)),\qquad (x,y)\in X_2\times Y_2
\end{equation*}
for every $F\in \lip(X_1,C(Y_1))$. We prove $\varphi(\cdot,y):X_2\to X_1$ is an isometry for every $y\in Y_2$. Let $x_1,x_2\in X_2$ and $y\in Y_2$ be arbitrary. Let $\beta$ with $\alpha<\beta<1$ be arbitrary. Set $f^\beta:X_1\to {\mathbb C}$ by $f^\beta(x)=d_1(x,\varphi(x_2,y))^\beta$. We have 
\begin{multline}\label{pisoLip3}
\frac{|f^\beta(s)-f^\beta(t)|}{d_1(s,t)^\alpha}=
\frac{|d_1(s,\varphi(x_2,y))^\beta-d_1(t,\varphi(x_2,y))^\beta|}{d_1(s,t)^\alpha}\\
\le 
\frac{d_1(s,t)^\beta}{d_1(s,t)^\alpha}=d_1(s,t)^{\beta-\alpha},\quad s,t\in X_1.
\end{multline}
Since $X_1$ is compact we have $\sup_{s,t\in X_1}d_1(s,t)<\infty$. Put $M=\sup_{s,t\in X_1}d_1(s,t)$. 
Then by \eqref{pisoLip3} we infer that $L_\alpha(f^\beta\otimes 1)\le M^{\beta-\alpha}$. 
We also infer by \eqref{pisoLip3} that $\lim_{s\to t}\frac{|f^\beta(s)-f^\beta(t)|}{d_1(s,t)^\alpha}=0$. Hence we have $f^\beta\otimes 1\in \lip(X_1,C(Y_1))$. We have, as before, 
\begin{equation}\label{pisoLip4}
\begin{split}
d_1(\varphi(x_1,y),\varphi(x_2,y))^\beta 
& = |f^\beta\otimes 1(\varphi(x_1,y),\tau(y))-f^\beta\otimes 1(\varphi(x_2,y),\tau(y))|
\\
& = |(U(f^\beta\otimes 1)(x_1,y)-(U(f^\beta\otimes 1)(x_2,y)|
\\
& \le L_\alpha(U(f^\beta\otimes 1))d_2(x_1,x_2)^\alpha
\\
& = L_\alpha(f^\beta\otimes 1)d_2(x_1,x_2)^\alpha
= M^{\beta-\alpha}d_2(x_1,x_2)^\alpha.
\end{split}
\end{equation}
Letting $\beta\to \alpha$ we have by \eqref{pisoLip4} that $d_1(\varphi(x_1,y),\varphi(x_2,y))^\alpha\le d_2(x_1,x_2)^\alpha$, hence $d_1(\varphi(x_1,y),\varphi(x_2,y))\le d_2(x_1,x_2)$. Applying the same argument for $U^{-1}$ as in the case of $\Lip(X_j,C(Y_j))$ we get
\[
d_2(x_1,x_2)^\beta\le M'^{\beta-\alpha}d_1(\varphi(x_1,y),\varphi(x_2,y))^\alpha
\]
for every $\beta$ with $\alpha<\beta<1$, where $M'=\sup_{s,t\in X_2}d_2(s,t)$. Letting $\beta\to \alpha$ we get $d_2(x_1,x_2)^\alpha \le d_1(\varphi(x_1,y),\varphi(x_2,y))^\alpha$ and $d_2(x_1,x_2) \le d_1(\varphi(x_1,y),\varphi(x_2,y))$. It follows that $d_2(x_1,x_2)=d_1(\varphi(x_1,y),\varphi(x_2,y))$ for every pair $x_1,x_2\in X_2$ and $y\in Y_2$, that is, $\varphi(\cdot,y)$ is an isometry for every $y\in Y_2$. 
\end{proof}
Note that if $Y_j$ is a singleton in Corollary \ref{isoLip}, then 
 $\Lip(X_j,C(Y_j))$ {\rm (}resp. $\lip(X_j,C(Y_j))${\rm )} is naturally identified with $\Lip(X_j)$ {\rm (}resp. $\lip(X_j)${\rm )}. 
Then Corollary \ref{isoLip} states that the statement of  Example 8 of \cite{jp} is indeed correct.
\begin{cor}\cite[Example 8]{jp} \label{JPOK}
The map $U:\Lip(X_1)\to \Lip(X_2)$ {\rm (}resp. $U:\lip(X_1)\to \lip(X_2)${\rm )} is a surjective isometry with respect to the norm $\|\cdot\|=\|\cdot\|_{\infty}+L(\cdot)$ {\rm (}resp. $\|\cdot\|=\|\cdot\|_{\infty}+L_\alpha(\cdot)${\rm )} if and only if there exists a complex number $c$ with the unit modulus and a surjective isometry $\varphi:X_2\to X_1$ such that
\[
U(F)(x)=cF(\varphi(x)), \qquad x\in X_2
\]
for every $F\in \Lip(X_1)$ {\rm (}resp. $F\in \lip(X_1)${\rm )}.
\end{cor}
\begin{proof}
Suppose that $U$ is a surjective isometry, then by Corollary \ref{isoLip} there exists a complex number $c$ with the unit modulus and a surjective isometry $\varphi:X_2\to X_1$ such that the desired equality holds.

Suppose that $c$ is a complex number with the unit modulus and $\varphi:X_2\to X_1$ is a surjective isometry. Then $U:\Lip(X_1) \to \Lip(X_2)$ (resp. $U:\lip(X_1)\to \lip(X_2)$) by $U(F)(x)=cF(\varphi(x))$, $x\in X_2$ for $F\in \Lip(X_1)$ (resp. $F\in \lip(X_1)$) is well defined. Then by Corollary \ref{isoLip} we have that $U$ is a surjective isometry. 
\end{proof}
\begin{example}\label{C101n}
Let $Y$ be a compact Hausdorff space. Then 
\[
([0,1], C(Y), C^1([0,1]), C^1([0,1],C(Y)))
\]
is an admissible quadruple of type L, where the norm of $f\in C^1([0,1])$ is defined by $\|f\|=\|f\|_\infty+\|f'\|_\infty$ and the norm of $F\in C^{1}([0,1],C(Y))$ is defined by $\|F\|=\|F\|_\infty+\|F'\|_\infty$. It is easy to see that $C^1([0,1])\otimes C(Y)\subset C^1([0,1],C(Y))$ and 
\[
\{F(\cdot, y):F\in C^1([0,1],C(Y)),\,\,y\in Y\}\subset C^1([0,1]).
\]
Let $\mathfrak{M}=[0,1]\times Y$ and $D:C^1([0,1],C(Y))\to C(\mathfrak{M})$ be defined by $D(F)(x,y)=F'(x,y)$ for $F\in C^1([0,1],C(Y))$. Then  $\|F'\|_\infty=\|D(F)\|_{\infty}$ for $F\in C^1([0,1],C(Y))$. Then the conditions from $\cno$ through $\cnt (3)$  of Definition \ref{aqL} are satisfied.
\end{example}
\begin{example}\label{C1T}
Let $Y$ be a compact Hausdorff space. Then 
\[
(\mathbb{T}, C(Y), C^1(\mathbb{T}), C^1(\mathbb{T},C(Y)))
\]
is an admissible quadruple of type L, where the norm of $f\in C^1(\mathbb{T})$ is defined by $\|f\|=\|f\|_\infty+\|f'\|_\infty$ and the norm of $F\in C^{1}(\mathbb{T},C(Y))$ is defined by $\|F\|=\|F\|_\infty+\|F'\|_\infty$. It is easy to see that $C^1(\mathbb{T})\otimes C(Y)\subset C^1(\mathbb{T},C(Y))$ and 
\[
\{F(\cdot, y):F\in C^1(\mathbb{T},C(Y)),\,\,y\in Y\}\subset C^1(\mathbb{T}).
\]
Let $\mathfrak{M}=\mathbb{T}\times Y$ and $D:C^1(\mathbb{T},C(Y))\to C(\mathfrak{M})$ be defined by $D(F)(x,y)=F'(x,y)$ for $F\in C^1(\mathbb{T},C(Y))$. Then $\|F'\|_\infty=\|D(F)\|_{\infty}$ for $F\in C^1(\mathbb{T},C(Y))$. Then the conditions from $\cno$ through $\cnt (3)$ of definition \ref{aqL} are satisfied for  $(\mathbb{T}, C(Y), C^1(\mathbb{T}), C^1(\mathbb{T},C(Y))$.
\end{example}

\begin{cor}\label{c101}
Let $Y_j$ be a compact Hausdorff space for $j=1,2$. The norm $\|F\|$ of $F\in C^{1}([0,1],C(Y_j))$ is defined by $\|F\|=\|F\|_\infty+\|F'\|_\infty$. Suppose that 
$U:C^1([0,1], C(Y_1))\to C^1([0,1],C(Y_2))$ is a map. Then $U$ is a surjective isometry if and only if 
there exists $h\in C(Y_2)$ such that $|h|=1$ on $Y_2$, a continuous map $\varphi:[0,1]\times Y_2\to [0,1]$ such that for each $y\in Y_2$ we have $\varphi(x,y)=x$ for every $x\in [0,1]$ or $\varphi(x,y)=1-x$ for every $x\in [0,1]$, and a homeomorphism $\tau:Y_2\to Y_1$ which satisfy that
\[
U(F)(x,y)=h(y)F(\varphi(x,y),\tau(y)),\qquad (x,y)\in [0,1]\times Y_2
\]
for every $F\in C^1([0,1],C(Y_1))$.
\end{cor}
\begin{proof}
Suppose that $U:C^1([0,1], C(Y_1))\to C^1([0,1],C(Y_2))$ is a surjective isometry. Then by Theorem \ref{main} there exists
 $h\in C(Y_2)$ such that $|h|=1$ on $Y_2$, a continuous map $\varphi:[0,1]\times Y_1\to [0,1]$ such that $\varphi(\cdot,y):[0,1]\to [0,1]$ is a homeomorphism for each $y\in Y_2$, and a homeomorphism $\tau:Y_2\to Y_1$ which satisfy
\begin{equation}\label{c1teq1}
U(F)(x,y)=h(y)F(\varphi(x,y),\tau(y)),\qquad (x,y)\in [0,1] \times Y_2
\end{equation}
for every $F\in C^{1}([0,1],C(Y_1))$. We only need to prove that, for every $y\in Y_2$ $\varphi(x,y)=x$ for every $x\in [0,1]$ or $\varphi(x,y)=1-x$ for every $x\in [0,1]$. Let $F_0\in C^1([0,1],C(Y_1))$ be defined by $F_0(x,y)=x$ for every $(x,y)\in [0,1]\times Y_1$. Then we have $F_0'=1$ on $[0,1]\times Y_1$ and $\|F_0\|=\|F_0\|_\infty+\|F_0'\|_\infty=2$. By \eqref{c1teq1} we have $U(F_0)(x,y)=h(y)\varphi(x,y)$ for every $(x,y)\in [0,1]\times Y_2$. 
Since $U(F_0)$ is continuously differentiable we infer that $\varphi$ is continuously differentiable and that $U(F_0)'(x,y)=h(y)\varphi'(x,y)$ for every $(x,y)\in [0,1]\times Y_2$. By \eqref{c1teq1} we infer that $\|U(F_0)\|_\infty=\|F_0\|_\infty$, hence $\|U(F_0)'\|_\infty=\|F_0'\|_\infty$ since $U$ is an isometry with respect to $\|\cdot\|$. As $|h|=1$ on $Y_2$ we see that 
\[
|\varphi'(x,y)|\le \|U(F_0)'\|_\infty=\|F_0'\|_\infty=1
\]
for every $(x,y)\in [0,1]\times Y_2$. We prove that $|\varphi'(x,y)|=1$ for every $(x,y)\in [0,1]\times Y_2$. Suppose contrary that there exists $(x_0,y_0)\in [0,1]\times Y_2$ with $|\varphi'(x_0,y_0)|<1$. As $\varphi(\cdot,y_0):[0,1]\to [0,1]$ is a homeomorphism we infer that $|\varphi(1,y_0)-\varphi(0,y_0)|=1$. As $\varphi(\cdot,y_0)$ is continuously differentiable we have
\[
1=|\varphi(1,y_0)-\varphi(0,y_0)|=|\int^1_0\varphi'(x,y_0)dx|\le
\int^1_0|\varphi'(x,y_0)|dx.
\]
Since $\varphi'$ is continuous and $|\varphi'|\le 1$ on $[0,1]\times Y_1$, and $|\varphi'(x_0,y_0)|<1$ we have 
\[
\int^1_0|\varphi'(x,y_0)|dx<1,
\]
which is a contradiction. Hence we have that $|\varphi'(x,y)|=1$ for every $(x,y)\in [0,1]\times Y_2$. Let $y_1\in Y_2$ be arbitrary. As $\varphi'(\cdot,y_1)$ is continuous on $[0,1]$ and $|\varphi'(\cdot,y_1)|=1$ on $[0,1]$ we have that 
$\varphi'(\cdot,y_1)=1$ on $[0,1]$ or $\varphi'(\cdot,y_1)=-1$ on $[0,1]$ since $\varphi'$ is real-valued with $|\varphi'|=1$ on a connected space $[0,1]$. It follows by a simple calculation that $\varphi(x,y_1)=x$ for every $x\in [0,1]$ or $\varphi(x,y_1)=1-x$ for every $x\in [0,1]$ since $\varphi(\cdot,y_1)$ is a bijection between $[0,1]$.

Suppose conversely that there exists $h\in C(Y_2)$ such that $|h|=1$ on $Y_2$, a continuous map $\varphi:[0,1]\times Y_2\to [0,1]$ such that for each $y\in Y_2$ $\varphi(x,y)=x$ for every $x\in [0,1]$ or $\varphi(x,y)=1-x$ for every $x\in [0,1]$, and a homeomorphism $\tau:Y_2\to Y_1$ which satisfy that
\[
U(F)(x,y)=h(y)F(\varphi(x,y),\tau(y)),\qquad (x,y)\in [0,1]\times Y_2
\]
for every $F\in C^1([0,1],C(Y_1))$. It is straightforward to check that $\|U(F)\|_\infty = \|F\|_\infty$. Let $y\in Y_2$ be arbitrary. By a simple calculation we infer that $|U(F)'(x,y)|=|F'(x,\tau(y))|$ for every $x\in [0,1]$ or $|U(F)'(x,y)|=|F'(1-x,\tau(y))|$ for every $x\in [0,1]$ for each $y\in Y_2$ and $F\in C^1([0,1],C(Y_1))$. As $\tau$ is a surjection, we have $\|U(F)'\|_\infty=\|F'\|_\infty$ for every $F\in C^1([0,1],C(Y_1))$. 
To prove that $U$ is surjective, let $F\in C^1([0,1],C(Y_2))$ be an arbitrary map. Put $G(x',y')=\overline{h(\tau^{-1}(y'))}F(\varphi(x',\tau^{-1}(y')),\tau^{-1}(y'))$, $(x',y')\in [0,1]\times Y_1$. It is easy to see that $G\in C^1([0,1],C(Y_1))$. As $\varphi(x,y)=x$ or $1-x$ depending on $y\in Y_2$ we see by a simple calculation that $\varphi(\varphi(x,y),y)=x$ for every $(x,y)\in [0,1]\times Y_2$. Then we have
\begin{multline*}
(U(G))(x,y)=h(y)G(\varphi(x,y),\tau(y))
\\
=
h(y)\overline{h(\tau^{-1}(\tau(y)))}F(\varphi(\varphi(x,y)),\tau^{-1}(\tau(y)),\tau^{-1}(\tau(y)))
\\
=F(\varphi(\varphi(x,y)),y)=F(x,y),\quad (x,y)\in [0,1]\times Y_2
\end{multline*}
It follows that $U$ is a surjective isometry from $C^1([0,1],C(Y_1))$ onto $C^1([0,1],C(Y_2))$.
\end{proof}
Note that if $Y_j$ is a singleton in Corollary \ref{c101}, then $C^1([0,1],C(Y_j))$ is $C^1([0,1],{\mathbb C})$. The corresponding result on isometries was given by Rao and Roy \cite{rr}.
\begin{cor}\label{c1t}
Let $Y_j$ be a compact Hausdorff space for $j=1,2$. The norm $\|F\|$ of $F\in C^{1}(\mathbb{T},C(Y_j))$ is defined by $\|F\|=\|F\|_\infty+\|F'\|_\infty$. Suppose that 
$U:C^1(\mathbb{T}, C(Y_1))\to C^1(\mathbb{T},C(Y_2))$ is a map. Then $U$ is a surjective isometry if and only if 
there exists $h\in C(Y_2)$ such that $|h|=1$ on $Y_2$, a continuous map $\varphi:\mathbb{T}\times Y_2\to \mathbb{T}$ and a continuous map $u:Y_2\to \mathbb{T}$ such that for every $y\in Y_2$ $\varphi(z,y)=u(y)z$ for every $z\in \mathbb{T}$ or $\varphi(z,y)=u(y)\bar{z}$ for every $z\in \mathbb{T}$, and a homeomorphism $\tau:Y_2\to Y_1$ which satisfy that
\[
U(F)(z,y)=h(y)F(\varphi(z,y),\tau(y)),\qquad (z,y)\in \mathbb{T}\times Y_2
\]
for every $F\in C^1(\mathbb{T},C(Y_1))$.
\end{cor}

\begin{proof}
Suppose that $U:C^1(\mathbb{T}, C(Y_1))\to C^1(\mathbb{T},C(Y_2))$ is a surjective isometry. Then by Theorem \ref{main} there exists $h\in C(Y_2)$ such that $|h|=1$ on $Y_2$, a continuous map $\varphi:\mathbb{T}\times Y_1\to \mathbb{T}$ such that $\varphi(\cdot,y):\mathbb{T}\to \mathbb{T}$ is a homeomorphism for each $y\in Y_2$, and a homeomorphism $\tau:Y_2\to Y_1$ which satisfy
\begin{equation}\label{c1teq1.5}
U(F)(z,y)=h(y)F(\varphi(z,y),\tau(y)),\qquad (z,y)\in \mathbb{T}\times Y_2
\end{equation}
for every $F\in C^{1}(\mathbb{T},C(Y_1))$. We prove that for every $y\in Y_2$ there corresponds $u(y)\in \mathbb{T}$ such that $\varphi(z,y)=u(y)z$ for every $z\in \mathbb{T}$ or $\varphi(z,y)=u(y)\bar{z}$ for every $z\in \mathbb{T}$. Let $F_0\in C^1(\mathbb{T},C(Y_1))$ be defined as $F_0(z,y)=z$ for every $(z,y)\in \mathbb{T}\times Y_1$. Then by \eqref{c1teq1.5} we have $U(F_0)(z,y)=h(y)\varphi(z,y)$. As $|h|=1$ on $Y_2$ we have that $\varphi=\bar{h}U(F_0)\in C^1(\mathbb{T},C(Y_2))$. We also have $\|F_0\|_\infty=1$ and $\|F_0'\|_\infty=1$, hence $\|F_0\|=2$. By \eqref{c1teq1.5} we have $\|U(F_0)\|_\infty=1$. Since $\|U(F_0)\|=\|F_0\|$, we infer that $\|U(F_0)'\|_\infty=\|F_0'\|_\infty$, where
\[
U(F_0)'(z,y)
=h(y)\varphi'(z,y), \quad (z,y)\in \mathbb{T}\times Y_2
\]
as $U(F_0)=h\varphi$. Thus 
\[
\|\varphi'\|_\infty=\|U(F_0)'\|_\infty=\|F_0'\|_\infty=1.
\]
It follows that $|\varphi'(z,y)|\le 1$ for every $(z,y)\in \mathbb{T}\times Y_2$. Define $u:Y_2\to {\mathbb T}$ by $u(y)=\varphi(1,y)$. Then $u$ is continuous since $\varphi$ is continuous on ${\mathbb T}\times Y_2$. We also have that $|u(y)|=|\varphi(1,y)|=1$. As $\varphi(\cdot,y)$ is a bijection from $\mathbb{T}$ onto itself, we have $\varphi(\mathbb{T}\setminus\{1\},y)=\mathbb{T}\setminus\{u(y)\}$. Hence the map
\[
t\mapsto -i\Log \overline{u(y)}\varphi(e^{it},y)
\]
is well defined from $(0,2\pi)$ onto $(0,2\pi)$, where $\Log$ denotes the principal value of the logarithm. As $\varphi(\cdot,y)$ is continuously differentiable, the above map has a natural extension $\mathcal{L}:[0,2\pi]\to [0,2\pi]$ (defining by $\mathcal{L}(0)=0$ and $\mathcal{L}(2\pi)=2\pi$, or $\mathcal{L}(0)=2\pi$ and $\mathcal{L}(2\pi)=0$, $\mathcal{L}(t)=-i\Log \overline{u(y)}\varphi(e^{it},y)$ for $0<t<2\pi$), which is continuously differentiable. By a simple calculation we have
\[
\mathcal{L}'(t)=\frac{\varphi'(e^{it},y)e^{it}}{\varphi(e^{it},y)}, \quad t\in [0,2\pi].
\]
Hence $|\mathcal{L}'(t)|\le 1$ for every $t\in [0,2\pi]$ since $|\varphi'(z,y)|\le 1$ for every $(z,y)\in \mathbb{T}\times Y_2$. In the way as in the proof of Corollary \ref{c101} we have that $\mathcal{L}'=1$ on $[0,2\pi]$ or $\mathcal{L}'=-1$ on $[0,2\pi]$. It follows that $\overline{u(y)}\varphi(e^{it},y)=e^{it}$ for every $t\in [0,2\pi]$ or $\overline{u(y)}\varphi(e^{it},y)=e^{-it}$ for every $t\in [0,2\pi]$. Hence $\varphi(z,y)=u(y)z$ for every $z\in \mathbb{T}$ or $\varphi(z,y)=u(y)\bar{z}$ for every $z\in \mathbb{T}$. 

Suppose conversely that there exists $h\in C(Y_2)$ such that $|h|=1$ on $Y_2$, a continuous map $\varphi:\mathbb{T}\times Y_2\to \mathbb{T}$ and a continuous map $u:Y_2\to \mathbb{T}$ such that $\varphi(z,y)=u(y)z$ for every $z\in \mathbb{T}$ or $\varphi(z,y)=u(y)\bar{z}$ for every $z\in \mathbb{T}$, and a homeomorphism $\tau:Y_2\to Y_1$ which satisfy that
\begin{equation}\label{c1teq2}
U(F)(z,y)=h(y)F(\varphi(z,y),\tau(y)),\qquad (z,y)\in \mathbb{T}\times Y_2
\end{equation}
for every $F\in C^1(\mathbb{T},C(Y_1))$. By the hypotheses on $\varphi$ and $\tau$ we infer that $(z,y)\mapsto (\varphi(z,y),\tau(y))$ gives a homeomorphism from $\mathbb{T}\times Y_2$ onto $\mathbb{T}\times Y_1$. As $|h|=1$ on $Y_2$ we infer that $\|F\|_\infty=\|U(F)\|_\infty$ for every $F\in C^1(\mathbb{T},C(Y_1))$. By \eqref{c1teq2} we have
\[
U(F)'(z,y)=h(y)F'(\varphi(z,y),\tau(y))\varphi'(z,y),\quad (z,y)\in \mathbb{T}\times Y_2
\]
for every $F\in C^1(\mathbb{T},C(Y_1))$. As $\varphi'(z,y)=u(y)$ on $\mathbb{T}\times Y_2$ or $\varphi'(z,y)=-u(y)\bar{z}^2$ on $\mathbb{T}\times Y_2$ we infer that
\[
\|U(F)'\|_\infty=\|hF'(\varphi,\tau)\varphi'\|_\infty=\|F'\|_\infty.
\]
It follows that $U$ is an isometry. It is not difficult to prove that $U$ is a surjection. We conclude that $U$ is a surjective isometry.
\end{proof}
\subsection*{Acknowledgements}
The authors record our sincerest appreciation to the two referees for their valuable comments and advice which have improved the presentation of this paper substantially.
The first author was partially supported by JSPS KAKENHI Grant Numbers JP16K05172 (representative), JP15K04921 (sharer), JP15K04897 (sharer), Japan Society for the Promotion of Science. 


\begin{thebibliography}{99}
\bibitem{araduba}
J.~Araujo and L.~Dubarbie,
\emph{
Noncompactness and noncompleteness in isometries of Lipschitz spaces},
J. Math. Anal. Appl.
{\bf 377} (2011), 15--29.

\bibitem{bfj}
F.~Botelho, R.~J.~Fleming and J.~E.~Jamison,
\emph{
Extreme points and isometries on vector-valued Lipschitz spaces},
J. Math. Anal. Appl. 
{\bf 381} (2011), 821--832.

\bibitem{bjStudia2009}
F.~Botelho and J.~Jamison,
\emph{
Surjective isometries on spaces of differentiable vector-valued functions},
Studia Math.
{\bf 192} (2009), 39--50.

\bibitem{bjRocky}
F.~Botelho and J.~Jamison,
\emph{
Homomorphisms on a class of commutative Banach algebras},
Rocky Mountain J. Math.
{\bf 43} (2013), 395--416.

\bibitem{bjPositivity17}
F.~Botelho and J.~Jamison,
\emph{
Surjective isometries on spaces of vector valued continuous and Lipschitz functions},
Positivity 
{\bf 17} (2013), 395--405; 
\emph{
Erratum to: Surjective isometries on spaces of vector valued continuous and Lipschitz functions} by F.~Botelho, 
{\bf 20} (2016), 757--759.

\bibitem{bjz}
F.~Botelho, J.~Jamison and B.~Zheng,
\emph{
Isometries on spaces of vector valued Lipschitz functions},
{\bf 17} (2013), 47--65.

\bibitem{c}
M.~Cambern,
\emph{
Isometries of certain Banach algebras},
 Studia Math. 
{\bf 25} (1964/1965),  217--225.

\bibitem{dl}
K.~de Leeuw,
\emph{
Banach spaces of Lipschitz functions},
Studia Math.
{\bf 21} (1961), 55--66.

\bibitem{fj1}
R.~J.~Fleming and J.~E.~Jamison,
\emph{
Isometries on Banach spaces: function spaces},
 Chapman \& Hall/CRC Monographs and Surveys in Pure and Applied Mathematics, 129. Chapman \& Hall/CRC, Boca Raton, FL, 2003. x+197 pp. ISBN: 1-58488-040-6.

\bibitem{fj2}
R.~J.~Fleming and J.~E.~Jamison,
\emph{
Isometries on Banach spaces. Vol. 2. Vector-valued function spaces},
 Chapman \& Hall/CRC Monographs and Surveys in Pure and Applied Mathematics, 138. Chapman \& Hall/CRC, Boca Raton, FL, 2008. x+234 pp. ISBN: 978-1-58488-386-9; 1-58488-386-3.

\bibitem{ho}
O.~Hatori and S.~Oi,
\emph{
Hermitian operators on Banach algebras of vector-valued Lipschitz maps},
J. Math. Anal. Appl. 
{\bf 452} (2017), 378--387;
\emph{Corrigendum to "Hermitian operators on Banach algebras of vector-valued Lipschitz maps'' [J. Math. Anal. Appl. 452 (2017) 378--387] [MR3628025].
}
 J. Math. Anal. Appl. 
{\bf 453} (2017), 1151--1152.

\bibitem{hots}
O.~Hatori, S.~Oi and H.~Takagi,
\emph{
Peculiar homomorphisms on algebras of vector valued  maps},
Studia Math.
Online first, DOI: 10.4064/sm8799-6-2017

\bibitem{ja}
K.~Jarosz, 
\emph{Isometries in semisimple, commutative Banach algebras}, Proc. Amer. Math. Soc. 
{\bf 94} (1985), 65--71.

\bibitem{jp}
K.~Jarosz and V.~D.~Pathak,
\emph{
Isometries between function spaces},
Trans. Amer. Math. Soc.
{\bf 305} (1988), 193--206.

\bibitem{jlp}
A.~Jim\'enez-Vargas, L.~Li, A.~M.~Peralta, L.~Wang and Y.-S~Wang,
\emph{
2-local standard isometries on vector-valued Lipschitz function spaces},
arXiv:1708.0289v1, 9 Aug 2017.

\bibitem{amHouston}
A.~Jim\'enez-Vargas and M.~Villegas-Vallecillos,
\emph{
Into linear isometries between spaces of Lipschitz functions},
Houston Journal Math. 
{\bf 34} (2008), 1165--1184.

\bibitem{amPAMS}
A.~Jim\'enez-Vargas and M.~Villegas-Vallecillos,
\emph{
Linear isometries between spaces of vector-valued Lipschitz functions},
 Proc. Amer. Math. Soc. 
{\bf137} (2009), 1381--1388.

\bibitem{amy}
A.~Jim\'enez-Vargas, M.~Villegas-Vallecillos and Y.-S.~Wang,
\emph{
Banach--Stone theorems for vector-valued little Lipschitz functions},
 Publ. Math. Debrecen 
{\bf 74} (2009),  81--100.

\bibitem{kawar}
K.~Kawamura,
\emph{
Banach--Stone type theorem for $C^1$-function spaces over Riemannian manifolds},
to appear in Acta Sc. Math. (Szeged).


\bibitem{kc1}
K.~Kawamura,
\emph{
Perturbations of norms on $C^1$-function spaces and associated isometry groups},
to appear in Top. Proc.

\bibitem{kc12}
K.~Kawamura,
\emph{
A Banach--Stone type theorem for $C^1$-function spaces over the circle}
submitted, 2017.

\bibitem{kkm}
K.~Kawamura, H.~Koshimizu and T.~Miura,
\emph{
Norms on $C^1([0,1])$ and there isometries},
preprint, 2017.

\bibitem{kos}
H.~Koshimizu,
\emph{
Linear isometries on spaces of continuously differentiable and Lipschitz continuous functions},
Nihonkai Math. J. 
{\bf 22} (2011), 73--90.

\bibitem{lcmw}
L.~Li, D.~Chen, Q.~Meng and Y.-S.~Wang,
\emph{
Surjective isometries on vector-valued differentiable function spaces},
preprint 2017.

\bibitem{lpww}
L.~Li, A.~M.~Peralta, L.~Wang and Y.-S.~Wang,
\emph{
Weak-2-local isometries on uniform algebras and Lipschitz algebras},
arXiv:1705.03619v1, 10 May 2017.

\bibitem{mw}
E.~Mayer-Wolf,
\emph{
Isometries between Banach spaces of Lipschitz functions},
Israel J. Math.
{\bf 38} (1981), 58--74.


\bibitem{mt}
T.~Miura and H.~Takagi,
\emph{
Surjective isometries on the Banach space of continuously differentiable functions},
Contemp. Math. 
{\bf 687} (2017), 181--192.

\bibitem{no}
A.~Nikou and A.~G.~O'Farrell,
\emph{
Banach algebras of vector-valued functions},
Glasgow Math. J.
{\bf 56} (2014), 419--426.

\bibitem{pal}
T.~W.~Palmer,
\emph{
Banach algebras and the general theory of $*$-algebras. Vol. 1. Algebras and Banach algebras},
Encyclopedia of Mathematics and its Applications, 
Cambridge University Press, Cambridge, 1994.


\bibitem{ph}
R.~R.~Phelps, 
\emph{
Lectures on Choquet's Theorem. Second edition},
Lecture Notes in Mathematics, 1757. Springer-Verlag, Berlin, 2001. viii+124 pp. ISBN: 3-540-41834-2.

\bibitem{rm}
A.~Ranjbar-Motlagh,
\emph{
A note on isometries of Lipschitz spaces},
J. Math. Anal. Appl. 
{\bf 411} (2014), 555--558.

\bibitem{rr}
N.~V.~Rao and A.~K.~Roy,
\emph{
Linear isometries of some function spaces},
Pacific J. Math. 
{\bf 38} (1971), 177--192.

\bibitem{roy}
A.~K.~Roy,
\emph{
Extreme points and linear isometries of the Banach space of Lipschitz functions},
Canad. J. Math. 
{\bf 20} (1968), 1150--1164.


\bibitem{wea}
N.~Weaver,
\emph{
Isometries of noncompact Lipschitz spaces},
Canad. Math. Bull.
{\bf 38} (1995), 242--249.
\end{thebibliography}
\end{document}